\documentclass[12pt]{amsart}
\usepackage{amssymb}
\usepackage{amsmath,latexsym}
\usepackage{amsfonts}
\usepackage{amsthm}
\usepackage[mathscr]{eucal}
\usepackage{mathrsfs}

\allowdisplaybreaks[4]

\newcommand{\rr}{\mathbb{R}}

\newcommand{\zz}{\mathbb{Z}}
\newcommand{\cc}{\mathbb{C}}
\newcommand{\hh}{\mathbb{H}}
\newcommand{\pp}{\mathbb{P}}
\newcommand{\Vol}{{\rm Vol}}
\newcommand{\vol}{{\rm vol}}
\newcommand{\I}{\mbox{${\mathbb I}$}}

\newcommand{\tr}{{\rm tr}}

\newcommand{\SO}{{\rm SO}}
\newcommand{\SU}{{\rm SU}}
\newcommand{\Sp}{{\rm Sp}}
\newcommand{\Spin}{{\rm Spin}}
\newcommand{\U}{{\rm U}}

\newcommand{\so}{\mbox{${\mathfrak s \mathfrak o}$}}

\newcommand{\g}{\mbox{${\mathfrak g}$}}

\newcommand{\kf}{\mbox{${\mathfrak k}$}}

\newcommand{\m}{\mbox{${\mathfrak m}$}}

\numberwithin{equation}{section}

\newtheorem{thm}{Theorem}[section]
\newtheorem{defn}[thm]{Definition}
\newtheorem{lem}[thm]{Lemma}
\newtheorem{prop}[thm]{Proposition}
\newtheorem{cor}[thm]{Corollary}

\newtheorem{remark}[thm]{Remark}
\newtheorem{ex}[thm]{Example}

\newenvironment{example}{\begin{ex} \em}{\end{ex}}
\newenvironment{rmk}{\begin{remark} \em}{\end{remark}}

\begin{document}

\title[]{Instability of some Riemannian manifolds with real Killing spinors}

\author{Changliang Wang}
\address{Max-Planck-Institut f\"ur Mathematik, Vivatsgasse 7, Bonn 53111, Germany}
\email{cwangmath@mpim-bonn.mpg.de}

\author{M. Y.-K. Wang}
\address {Department of Mathematics and Statistics, McMaster University, Hamilton, Ontario, L8S4K1 Canada}
\email{wang@mcmaster.ca}



\date{revised \today}

\begin{abstract}
 We prove the instability of some families of Riemannian manifolds with non-trivial real Killing spinors.
These include the invariant Einstein metrics on the Aloff-Wallach spaces $N_{k, l}=\SU(3)/i_{k, l}(S^{1})$ (which are
all nearly ${\rm G}_2$ except $N_{1,0}$), and  Sasaki Einstein circle bundles over certain irreducible Hermitian symmetric spaces.
We also prove the instability of most of the simply connected non-symmetric compact homogeneous Einstein spaces of
dimensions $5, 6, $ and $7$, including the strict nearly K\"ahler ones (except ${\rm G}_2/\SU(3)$).
\end{abstract}

\maketitle

\section{\bf Introduction}

In this article we will derive the instability of some families of simply connected closed Einstein manifolds most of which admit
a non-trivial real Killing spinor. One consequence of our work is the existence of examples of unstable Einstein manifolds with
non-trivial real Killing spinors whose Euclidean metric cones realize all the possible irreducible special holonomy types.

Recall that for a spin manifold $(M^n, g)$, a Killing spinor $\sigma$ is a section of the complex spinor bundle which satisfies the equation
$$  \nabla_X \sigma = c \, X \cdot \sigma $$
for all tangent vectors $X$, where $\nabla$ is the spinor connection induced by the Levi-Civita connection of $g$, $\cdot$ denotes
Clifford multiplication, and $c$ is a priori a complex constant. By the fundamental work of T. Friedrich and his colleagues it is
now well-known that $c$, if nonzero, is either purely imaginary or real. Furthermore, the metric $g$ must be Einstein with
Einstein constant $\Lambda = 4c^2 (n-1)$.

In the $c=0$ case, the metric $g$ has restricted holonomy properly contained in $\SO(n)$. Calabi-Yau, hyperk\"ahler, torsion free
${\rm G}_2$ and ${\rm Spin}(7)$ manifolds belong to this class. These Einstein manifolds all turn out to be stable by the work of
Dai-Wang-Wei \cite{DWW05}. When $c$ is purely imaginary and $(M, g)$ is complete, the classification was achieved by H. Baum \cite{Bau89},
and proofs of the stability of the Einstein metrics were given, first by Kr\"oncke in \cite{Kr17}, and later
by the first author in \cite{Wan17}.

In the real case, an important conceptual classification was given by \cite{Ba93}, which can be summarized by the
statement that $(M, g)$ admits a non-trivial real Killing spinor iff its Euclidean metric cone admits a
non-trivial parallel spinor. Of course the detailed classification of these manifolds includes the study of Sasaki Einstein manifolds
(see e.g. \cite{BFGK91}, \cite{BG08}) in odd dimensions and nearly K\"ahler $6$-manifolds (see e.g. \cite{FH17}).
Furthermore, T. Friedrich \cite{Fr80} gave a lower bound for the eigenvalues of the Dirac operator on closed spin manifolds with positive
scalar curvature that depended on the dimension and minimum value of the scalar curvature. This result was later generalized in
\cite{Hi86} where the positivity (resp. minimum value) of the scalar curvature was replaced by the positivity
(resp. value) of the first eigenvalue of the conformal Laplacian. In both cases, the equality case is characterized
by manifolds admitting a non-trivial real Killing spinor. By comparison, the equality case of the Lichnerowicz estimate
for the first eigenvalue of the Laplace-Beltrami operator on a manifold with positive Ricci curvature is characterized by the
round spheres, which are stable and happen also to have a maximal family of Killing spinors.

In addition to their intrinsic interest within Differential Geometry, manifolds with real Killing spinors are of great interest
in Mathematical Physics. In the 1980s, such manifolds, especially ones of dimension six or seven, were independently investigated
by theoretical physicists in their pursuit of Kaluza-Klein compactifications in supergravity theories \cite{DNP86}.
More recently, interest in these spaces from the physics community stems from the AdS/CFT correspondence, see e.g., \cite{GMSW05}.
For all the above reasons, there is good motivation to study the stability problem for manifolds admitting
real Killing spinors.

We shall actually use various different notions of stability for Einstein metrics in this paper.
These are different from notions of stability used by physicists, see e.g., \cite{DNP86}, \cite{GiHa02},
\cite{GiHaP03} .  All Einstein manifolds under consideration hereafter will have positive Einstein constant.
Unless otherwise stated, we will exclude the case of round spheres.

The first stability notion comes from the fact that for a closed manifold $M^{n}$  Einstein metrics
are precisely the critical points of the normalized total scalar curvature functional
\begin{equation}
\widetilde{\bf S}(g)=\frac{1}{(\Vol(M, g))^{\frac{n-2}{n}}}\int_{M}s_{g}\,d\vol_{g}
\end{equation}
where in the above $s_{g}$ is the scalar curvature of the Riemannian metric $g$ on $M^{n}$. Since this functional is invariant
under the action of the diffeomorphisms of $M$ and is locally minimizing along conformal change directions, it is customary to
restrict $\widetilde{\bf S}$ to the space of Riemannian metrics with constant scalar curvature and fixed volume. The tangent space
to this ILH-manifold consists of the {\em TT-tensors}, i.e., symmetric $2$-tensors satisfying $\tr_g(h) = 0 $ and ${\delta}_g h = 0$
(\cite{Bes87}, section 4.G). The second variation of $\widetilde{\bf S}$ is then given by
\begin{equation}  \label{2ndvar-Hilbert}
\widetilde{\bf S}^{\prime\prime}_{g}(h, h)=\frac{-1}{2(\Vol(M, g))^{\frac{n-2}{n}}}\int_{M}\langle\nabla^{*}\nabla h-2\mathring{R}h, h\rangle \,d\vol_{g}
\end{equation}
where $(\mathring{R}h)_{ij}$ is defined to be $ \sum_{ k, l} \,R_{ikjl}h^{kl}$ and our convention for the curvature tensor is
$R_{X, Y} = \nabla_{[X, Y]} - [\nabla_X, \nabla_Y ]$.

\begin{defn}  \label{stability-def}
A closed Einstein manifold $(M, g)$ is

\noindent{$($a$)$} $\widetilde{\bf S}$-stable if $g$ is a local maximum of $\widetilde{\bf S}$
restricted to the space of Riemannian metrics on $M$ with constant scalar curvature and the same volume as $g$;

\noindent{$($b$)$} $\widetilde{\bf S}$-linearly stable if
     $\langle \nabla^* \nabla h - 2 \mathring{R} h, h \rangle_{L^2(M, g)}  \geq 0 $ for all
      TT-tensors $h$ on $M$.

\end{defn}

Note that $\widetilde{\bf S}$-stability is the first notion of stability mentioned by Koiso  (\cite{Koi80}, p. 52),
while the second notion is  weaker than that given in Definition 2.7 there. Both notions of stability in the above include the
possibility of non-trivial (resp. infinitesimal) Einstein deformations (which may not be integrable in general).
In the first case, the value of the restricted functional would be unchanged, while in
the second case one would get an eigentensor of the Lichnerowicz Laplacian with eigenvalue equal to twice
the Einstein constant $\Lambda$, owing to the identity $\nabla^* \nabla - 2 \mathring{R} = -(\Delta_L + 2 \Lambda \I)$
for an Einstein manifold. The $\widetilde{\bf S}$-coindex of $g$ is the dimension (necessarily finite by elliptic theory)
of the maximal negative definite subspace for the quadratic form
\begin{equation}  \label{quadratic}
 {\mathscr Q}(h, h) := \langle \nabla^* \nabla h - 2 \mathring{R} h, h \rangle_{L^2(M, g)}.
\end{equation}

The corresponding notions of instability are given by negation. Hence $\widetilde{\bf S}$-linear instability
implies $\widetilde{\bf S}$-instability. Moreover, $\widetilde{\bf S}$-linear instability implies $\nu$-linear instability and further also $\nu$-instability (see Definition $\ref{nu-stability-def}$ below). Then by Theorem 1.3 in \cite{Kr15},
(since $\Lambda > 0$) it also implies that $g$ is {\em dynamically unstable} for the Ricci flow.

Another notion of stability comes from the $\nu$-entropy of Perelman. For detailed information about this
functional and its second order properties we refer the reader to \cite{Pe02}, \cite{CHI04}, \cite{CM12}, and \cite{CH15}.
For us the important facts about the $\nu$-entropy to recall are that its value is unchanged by the action of diffeomorphisms
and homotheties, it is monotonic increasing along Ricci flows, and its critical points consist of shrinking
gradient Ricci solitons (which include Einstein metrics with positive $\Lambda$).
At an Einstein metric $g$,  the second
variation of the $\nu$-entropy is given (up to a positive constant) by $- \frac{1}{2} {\mathscr Q}(h, h)$ on the subspace
$ \ker \tr_g \cap \ker \delta_g$ of TT-tensors. However, unlike the case of the $\widetilde{\bf S}$ functional,  the second
variation is no longer always positive on ${\mathscr C}(M)g$--the positive directions are given by eigenfunctions
 of the Laplacian of $g$ corresponding to eigenvalues less than $2\Lambda$.
(Our convention for eigenvalues is given by $\Delta \phi = - \lambda \phi$ with $\lambda \geq 0$.)

\begin{defn} \label{nu-stability-def}
A closed Einstein manifold $(M, g)$ with Einstein constant $\Lambda$ is

\noindent{$($a$)$} $\nu$-stable if $g$ is a local maximizer of the $\nu$-entropy;

\noindent{$($b$)$} $\nu$-linearly stable if the second variation of the $\nu$-entropy is negative semi-definite on
          ${\mathscr C}(M)g \oplus (\ker \tr_g \cap \ker \delta_g)$.
\end{defn}

By \cite{CH15}, $\nu$-linear stability is equivalent to ${\mathscr Q}(h, h) \geq 0$ for all
TT-tensors and $\lambda_1(M, g) \leq 2 \Lambda$. Consequently, there are two sources contributing to
$\nu$-linear instability, and $\widetilde{\bf S}$-linear instability implies $\nu$-linear instability.
We will discuss the restriction of the second variation of the $\nu$-entropy to ${\mathscr C}(M)g$ in
greater detail in section \ref{instab-conf}, where we deduce the $\nu$-linear instability of
some homogeneous Einstein metrics which admit real Killing spinors. Here we only note the interesting fact
that destablizing directions coming from conformal deformation by eigenfunctions necessarily deform a
homogeneous Einstein metric away from the space of homogeneous metrics.

Before stating our results on instability, we need to recall a few more facts about simply connected
manifolds admitting a non-trivial real Killing spinor. Because of our assumption of simple connectivity
and our exclusion of round spheres, such a manifold is de Rham irreducible and cannot be a symmetric space
(p. 35, \cite{BFGK91}, Theorem 13). By the results of B\"ar \cite{Ba93}, if its Euclidean cone has $\SU(m+1)$
holonomy, $m \geq 2$, then $(M^{2m+1}, g)$ is Sasaki Einstein. (The implicit scaling involved is
choosing $\Lambda = 2m$, whence $c = \pm \frac{1}{2}$.)
Conversely, a simply connected Sasaki Einstein manifold is spin \cite{Mo97} and admits non-trivial
real Killing spinors (\cite{FrK90}, Theorem 1). The dimension of the space of real Killing spinors
is $2$ and the chiral nature (i.e., whether or not both signs of $c$ occur) of these spinors depends
on the parity of $m$.

If the  Euclidean cone of $(M^{4m+3}, g)$ has $\Sp(m+1)$ holonomy, $m \geq 1$, then $(M^{4m+3}, g)$ is $3$-Sasakian.
In this case the dimension of the space of real Killing spinors is $m+2$ and only one sign of $c$ occurs
once the orientation is fixed \cite{W89}.

Finally, if the Euclidean cone has $\Spin(7)$ (resp. ${\rm G}_2$ ) holonomy, then by \cite{BFGK91} and \cite{Ba93}
$(M, g)$ has a nearly ${\rm G}_2$ (resp. a strict nearly K\"ahler) structure, and the space of real
Killing spinors has dimension $1$. The converse statements are proved respectively in
\cite{BFGK91} and \cite{Gr90}. We also refer to \cite{FKSM97} for nearly ${\rm G}_{2}$ structures and Killing spinors.

An interesting family of simply connected closed Riemannian manifolds which are nearly ${\rm G}_2$
are the Aloff-Wallach spaces $N_{k, l} = \SU(3)/U_{k,l}$, where $k,l$ are relatively prime integers
and $U_{k,l}$ is the circle ${\rm diag}(e^{2\pi i k\theta}, e^{2\pi i l \theta}, e^{-2 \pi i (k+l) \theta})$
in $\SU(3)$. It is well-known that these manifolds are spin and, up to isometry, they admit two $\SU(3)$-invariant
Einstein metrics \cite{W82}, \cite{CR84}, \cite{PP84}, \cite{KoV93}, \cite{Nik04}. Except for the spaces
$N_{1, -1}$ and $N_{1,1}$, all the Euclidean cone metrics of these $\SU(3)$-invariant Einstein metrics
have holonomy $\Spin(7)$ (see, \cite{CR84}, \cite{BFGK91}, and \cite{Ba93}). Topologically speaking, the $N_{k,l}$
exhibit infinitely many homotopy types, and there exists pairs which are homeomorphic but not diffeomorphic \cite{KS91}.

\begin{thm} \label{Aloff-Wallach}
The invariant Einstein metrics on the Aloff-Wallach manifolds $N_{k, l}$ described above are all
$\widetilde{\bf S}$-linearly unstable, and therefore, $\nu$-linearly unstable.
\end{thm}

The proof of this theorem will be given in subsections \ref{AWI} and \ref{AWII}. Some remarks about the two
exceptional Aloff-Wallach spaces are also given in section \ref{homog}.

\medskip

Let $(M^{2m+1}, g), m \geq 2$ be a closed simply connected Sasakian Einstein manifold. The Sasaki structure
is {\em regular} if the characteristic vector field generates a free circle action on $M$. In this case,
$M$ is a principal circle bundle over a Fano K\"ahler Einstein manifold $B$ such that the projection map
is a Riemannian submersion with totally geodesic fibres, and the Euler class of the bundle is a rational
multiple of the first Chern class of $B$. It follows from Corollary 1.7 in \cite{WW18} that
if the second Betti number of $B$ is greater than $1$, then $(M, g)$ is $\widetilde{\bf S}$-linearly unstable.

When $b_2(B) = 1$, we have $H^2(B; \zz) \approx \zz$ (since $H^2(B; \zz)$ is torsion free), so all principal
circle bundles over it are, up to a change in orientation in the fibers, quotients of the circle bundle corresponding
to one of the two indivisible classes in $H^2(B; \zz)$. The total spaces of these two circle bundles are diffeomorphic
and simply connected. The simplest examples of Fano K\"ahler Einstein manifolds with $b_2 =1$ are the irreducible
hermitian symmetric spaces of compact type. For complex projective space ${\cc \pp}^m$, the corresponding simply connected
regular Sasaki Einstein manifold over it is just $S^{2m+1}$ equipped with the constant curvature $1$ metric, which
is $\widetilde{\bf S}$-linearly stable. By contrast we have

\begin{thm} The following simply connected regular Sasaki Einstein manifold are $\nu$-linearly unstable from
conformal variations:

\noindent{$($a$)$} $\SO(p+2)/\SO(p), p \geq 3,$   circle bundle over the complex quadric $ \SO(p+2)/(\SO(p) \times \SO(2))$;

\noindent{$($b$)$} ${\rm E}_6/ \Spin(10)$, and ${\rm E}_7/{\rm E}_6$, which are respectively circle bundles over the
     hermitian symmetric spaces ${\rm E}_6/(\Spin(10 \cdot \U(1))$ and ${\rm E}_7/({\rm E}_6 \cdot \U(1))$;

\noindent{$($c$)$} $\SU(p+2)/(\SU(p) \times \SU(2)), p \geq 2$, a circle bundle over the complex Grassmannian
      $\SU(p+2)/{\rm S}(\U(p)\times \U(2))$.

Moreover, the Stiefel manifolds in $($a$)$ above are also $\widetilde{\bf S}$-linearly unstable, and for $k \geq 4$, $\Sp(k)/\SU(k)$,
      which are circle bundles over $\Sp(k)/\U(k)$, are  $\widetilde{\bf S}$-linearly unstable, and so $\nu$-linearly unstable.
\end{thm}

The proof of this theorem, including the dimensions of the destablizing eigenspaces, are given in
sections \ref{Stiefel} and \ref{instab-conf}.

The $\widetilde{\bf S}$-linear instability of $\Sp(k)/\SU(k)$, $k\geq4$, follows from Koiso's work in \cite{Koi80} and Corollary 6.1 in \cite{Wan17}. Indeed, by Koiso's calculations on page 68 and the table on page 70 in \cite{Koi80}, the Hermitian symmetric space $\Sp(k)/\U(k)$ of dimension $k^{2}+k$ is $\widetilde{\bf S}$-linearly unstable. Moreover, after rescaling the symmetric metric used in \cite{Koi80} so that the
new Einstein constant is $k^2+k+2$,  one finds that $\nabla^{*}\nabla-2\mathring{R}$ has a negative eigenvalue $-4\frac{k^{2}+k+2}{2(k+1)}=-2k-\frac{4}{k+1}<-8$ if $k\geq4$.

\smallskip

In the last section of this paper we discuss the stability of compact simply connected homogeneous Einstein
manifolds of dimension $\leq 7$ by putting together the results in this paper and \cite{WW18}
with classification results for these manifolds, and the work of \cite{Koi80} and \cite{CH15}. The results are
summarized as follows:

\begin{thm}  \label{lowdim}
Let $(M=G/K, g)$ be a compact simply connected Einstein manifold on which the semisimple connected
Lie group $G$ acts almost effectively by isometries and with isotropy group $K$. Assume that $(G, K)$ is not
a Riemannian symmetric pair and that $5 \leq \dim M \leq 7$. Assume further that $M \neq S^3 \times S^3$
or the isotropy irreducible space $\Sp(2)/\SU(2)$. Then $g$ is $\widetilde{\bf S}$-linearly unstable.
\end{thm}

\medskip

As mentioned at the beginning of the Introduction, it follows from all the above theorems and Corollary \ref{3-Sasakian}
that there are $\widetilde{\bf S}$-linearly unstable (and hence $\nu$-linearly unstable and dynamically unstable)
examples of manifolds admitting non-trivial real Killing spinors exhibiting all possible Euclidean metric cone special
holonomy types and in all admissible dimensions. By contrast, up to now, the only $\widetilde{\bf S}$-linearly stable
examples with non-trivial real Killing spinors are the constant curvature spheres. 

\medskip

\noindent{\bf Acknowledgements:}  The first author would like to thank Professors Xianzhe Dai and Guofang Wei for their interests and many helpful discussions. During the Fall term of 2017-2018, he was supported by a Fields Postdoctoral Fellowship. He thanks the Fields Institute for Research in Mathematical Sciences for the support.

The second author is partially supported by NSERC Discovery Grant No. OPG00009421.

Both authors thank Professors Stuart Hall and Thomas Murphy for their comments on an earlier version of the paper.


\section{\bf Instability of Einstein metrics on Aloff-Wallach spaces}  \label{AW}

In this section we will prove Theorem $\ref{Aloff-Wallach}$, i.e., deduce the $\widetilde{\bf S}$-linear instability of all invariant Einstein metrics on the Aloff-Wallach spaces $N_{k, l}=\SU(3)/U_{k,l}$, where $k, l$ are integers, and $U_{k,l}$ is the circle ${\rm diag}(e^{2\pi i k\theta}, e^{2\pi i l \theta}, e^{-2 \pi i (k+l) \theta})$ in $\SU(3)$. We will assume in addition that $k, l$ are
coprime, so that $N_{k, l}$ is simply connected, and remove diffeomorphic spaces by assuming that $k \geq l \geq 0$.
In the proof we will use the  explicit solutions in \cite{CR84} for the invariant Einstein equations on all Aloff-Wallach spaces except one invariant Einstein metric on $N_{1,0}$. Thus in \S2.1 and 2.2 we will follow the notation in \cite{CR84}, except that the parameters $\alpha, \beta, \gamma, \delta$ in $(\ref{InvariantMetric})$ below are $\frac{1}{\alpha^{2}}, \frac{1}{\beta^{2}}, \frac{1}{\gamma^{2}}, \frac{1}{\delta^{2}}$ in \cite{CR84}.

In \cite{CR84}, the Aloff-Wallach spaces $N_{k, l}$ are denoted instead by $N^{pq0}$ with co-prime integers $p, q$. The Lie algebra of the embedded circle subgroup is generated by $N$ as defined in (\ref{LieAlgebraBasis}) below. By comparing $N$ with the Lie algebra of $U_{k, l}$, one obtains the following relationship between $k, l$ and $p, q$:
\begin{equation}\label{pq-klRelation}
\begin{cases}
p=(k-l)c,\\
q=3(k+l)c,
\end{cases}
\end{equation}
for some proportionality constant $c$.

Our assumptions on $k, l$ translate into the conditions that $p, q\geq0, (p, q)=1,$ and $3p\leq q$.  Then the integer pairs $(k, l)$ and $(p, q)$ uniquely determine each other by $(\ref{pq-klRelation})$. In \cite{CR84}, more general spaces $N^{pqr}$ with integers $p, q$, and $r$ taken to be relatively prime were studied. These spaces have $N^{pq0}$ as their universal covers. We also note that the
special spaces $N_{1,1}, N_{1, 0}$ correspond respectively to $N^{010}$ and $N^{130}$. (These spaces are special because their isotropy
representations contain equivalent irreducible summands.)


\subsection{Einstein metrics on $N^{pq0}$}
In this subsection, we will recall the Einstein equations on the Aloff-Wallach spaces and their solutions
in \cite{CR84}. They will play important roles in \S $\ref{AWI}$.

We use the following basis and the decomposition of the Lie algebra $\mathfrak{su}(3)$ as in \cite{CR84}.
For each fixed pair of integers $p, q$ with $(p, q)=1$, let
\begin{equation}\label{LieAlgebraBasis}
\begin{aligned}
& N = -\frac{i}{\sqrt{3p^{2}+q^{2}}}\begin{bmatrix}
                                      -\frac{\sqrt{3}}{6}(q-3p) & 0 & 0 \\
                                      0 & -\frac{\sqrt{3}}{6}(q+3p) & 0 \\
                                      0 & 0 & \frac{\sqrt{3}}{3}q
                                      \end{bmatrix}, \\ \\
& Z = -\frac{i}{\sqrt{3p^{2}+q^{2}}}\begin{bmatrix}
                                      \frac{p+q}{2} & 0 & 0 \\
                                      0 & \frac{p-q}{2} & 0 \\
                                      0 & 0 & -p
                                      \end{bmatrix},\\ \\
& X_{1}=-\frac{1}{2}i\lambda_{1}=\begin{bmatrix}
                                0 & -\frac{1}{2}i & 0\\
                                -\frac{1}{2}i & 0 & 0\\
                                0 & 0 & 0
                                \end{bmatrix},
                                \qquad
X_{2}=-\frac{1}{2}i\lambda_{2}=\begin{bmatrix}
                              0 & -\frac{1}{2} & 0\\
                              \frac{1}{2} & 0 & 0 \\
                              0 & 0 & 0
                              \end{bmatrix},\\ \\
& X_{4}=-\frac{1}{2}i\lambda_{4}=\begin{bmatrix}
                                0 & 0 & -\frac{1}{2}i\\
                                0 & 0 & 0\\
                                -\frac{1}{2}i & 0 & 0
                                \end{bmatrix},
\qquad
X_{5}=-\frac{1}{2}i\lambda_{5}=\begin{bmatrix}
                              0 & 0 & -\frac{1}{2}\\
                              0 & 0 & 0 \\
                              \frac{1}{2} & 0 & 0
                              \end{bmatrix},\\ \\
& X_{6}=-\frac{1}{2}i\lambda_{6}=\begin{bmatrix}
                                0 & 0 & 0\\
                                0 & 0 & -\frac{1}{2}i\\
                                0 & -\frac{1}{2}i & 0
                                \end{bmatrix},
\qquad
X_{7}=-\frac{1}{2}i\lambda_{7}=\begin{bmatrix}
                              0 & 0 & 0\\
                              0 & 0 & -\frac{1}{2}\\
                              0 & \frac{1}{2} & 0
                              \end{bmatrix},
\end{aligned}
\end{equation}
where $\lambda_{k}$ are called the Gell-Mann matrices in the physics literature.
Let $\mathfrak{h}={\rm span}(N)$, $\mathfrak{m}_{1}={\rm span}(X_{1}, X_{2})$, $\mathfrak{m}_{2}={\rm span}(Z)$,
$\mathfrak{m}_{3}={\rm span}(X_{4}, X_{5})$, and $\mathfrak{m}_{4}={\rm span}(X_{6}, X_{7})$. Then we
have the Lie algebra decomposition:
\begin{equation}\label{LieAlgDec}
\mathfrak{su}(3)=\mathfrak{t}\oplus\mathfrak{m}_{1}\oplus\mathfrak{m}_{3}\oplus\mathfrak{m}_{4}
=\mathfrak{h}\oplus\mathfrak{m}_{2}\oplus\mathfrak{m}_{1}\oplus\mathfrak{m}_{3}\oplus\mathfrak{m}_{4}=\mathfrak{h}\oplus\mathfrak{m},
\end{equation}
where $\mathfrak{t}=\mathfrak{h}\oplus\mathfrak{m}_{2}$, and $\mathfrak{m}=\mathfrak{m}_{1}\oplus\mathfrak{m}_{2}\oplus\mathfrak{m}_{3}\oplus\mathfrak{m}_{4}$. Let $H_{p,q} \approx \U(1)$ be the isotropy group (generated by $N$) of the identity coset $[I_{3}]$, where $I_{3}$ denotes the identity matrix of size $3$. We identify the tangent space $T_{[I_{3}]}N^{pq0}$ with $\mathfrak{m}$  as usual.

The background metric chosen in \cite{CR84} is the negative of the Killing form of $\mathfrak{su}(3)$:
$Q(X, Y)$ $=-6\tr(XY)$ for any $X, Y\in \mathfrak{su}(3)$. This can be seen from (2.7) and the first equation in (2.5)
of \cite{CR84}. Then for any four positive real numbers $\alpha, \beta, \gamma, \delta$, the following ${\rm Ad}_{H}$-invariant
inner product
\begin{equation}\label{InvariantMetric}
g(\alpha, \beta, \gamma, \delta)=\alpha Q\vert_{\mathfrak{m}_{1}}\oplus\beta Q\vert_{\mathfrak{m}_{2}}
\oplus\gamma Q\vert_{\mathfrak{m}_{3}}\oplus\delta Q\vert_{\mathfrak{m}_{4}}
\end{equation}
on $\mathfrak{m}$ induces an $\SU(3)$-invariant Riemannian metric on $N^{pq0}$. We will also use $g(\alpha, \beta, \gamma, \delta)$ to denote this invariant Riemannian metric. In the generic cases, namely when $(p, q)\neq (0, 1)$ or $(1, 3)$, the components $\m_{1}$, $\m_{2}$, $\m_{3}$, and $\m_{4}$ of the decomposition of the isotropy representation $\m$ are inequivalent to each other. Thus the inner products in $(\ref{InvariantMetric})$  induce all of the $\SU(3)$-invariant metrics on $N^{pq0}$. On the other hand, it turns out that all Einstein metrics on $N^{010}$ and $N^{130}$ obtained in \cite{CR84} and \cite{PP84} also have the block diagonal form as in $(\ref{InvariantMetric})$. Thus, as in \cite{CR84}, we first only consider invariant metrics defined as in $(\ref{InvariantMetric})$.
 We remind the reader that the $\alpha, \beta, \gamma, \delta$ in $(\ref{InvariantMetric})$ are actually $\frac{1}{\alpha^{2}}, \frac{1}{\beta^{2}}, \frac{1}{\gamma^{2}}, \frac{1}{\delta^{2}}$ in \cite{CR84}.

The Ricci tensors of the metrics in $(\ref{InvariantMetric})$ have the same block diagonal form as the metrics. Their
 components are given in (2.13) in \cite{CR84}, and we recall them below.
\begin{equation}\label{RicciCurvatureAloffWallach}
\begin{split}
Ric\vert_{\m_{1}} &= \left[\frac{3}{4}\frac{1}{\alpha}+\frac{1}{8}\left(\frac{\alpha}{\gamma\delta}-\frac{\gamma}{\alpha\delta}-\frac{\delta}{\alpha\gamma}\right)
-\frac{1}{4}q^{2}\frac{\beta}{\alpha^{2}}\right]g\vert_{\m_{1}},\\
Ric\vert_{\m_{2}} &= \left[\frac{1}{4}q^{2}\frac{\beta}{\alpha^{2}}+\frac{1}{16}(3p+q)^{2}\frac{\beta}{\gamma^{2}}+\frac{1}{16}(3p-q)^{2}\frac{\beta}{\delta^{2}}\right]
g\vert_{\m_{2}},\\
Ric\vert_{\m_{3}} &=\left[\frac{3}{4}\frac{1}{\gamma}+\frac{1}{8}\left(\frac{\gamma}{\alpha\delta}-\frac{\alpha}{\gamma\delta}-\frac{\delta}{\alpha\gamma}\right)
-\frac{1}{16}(3p+q)^{2}\frac{\beta}{\gamma^{2}}\right]g\vert_{\m_{3}},\\
Ric\vert_{\m_{4}}&=\left[\frac{3}{4}\frac{1}{\delta}+\frac{1}{8}\left(\frac{\delta}{\alpha\gamma}-\frac{\alpha}{\gamma\delta}
-\frac{\gamma}{\alpha\delta}\right)-\frac{1}{16}(3p-q)^{2}\frac{\beta}{\delta^{2}}\right]g\vert_{\m_{4}}.
\end{split}
\end{equation}

Recall also the following change of variables given in (3.3) in \cite{CR84}:
\begin{equation}\label{ChangeVariables}
a=\frac{\delta}{\alpha},\quad b=\frac{\gamma}{\alpha}, \quad u=\sqrt{\frac{\beta\delta}{\alpha\gamma}}\frac{(3p+q)}{\sqrt{2}}, \quad v=-\sqrt{\frac{\beta\gamma}{\alpha\delta}}\frac{(3p-q)}{\sqrt{2}}, \quad \lambda=96\frac{\gamma\delta e^{2}}{\alpha},
\end{equation}
where we have used our choice of $\alpha, \beta, \gamma, \delta$  as in $(\ref{InvariantMetric})$.  Using this change of variables, Castellani and Romans transformed the Einstein equations for the metric $g(\alpha, \beta, \gamma, \delta)$  with Einstein constant $12e^{2}$ to the equations
\begin{equation}\label{EinsteinCondition}
\begin{aligned}
6ab+1-a^{2}-b^{2}-(av+bu)^{2} &=& \lambda,\\
6a+b^{2}-a^{2}-1-v^{2} &=& \lambda,\\
6b+a^{2}-b^{2}-1-v^{2} &=& \lambda,\\
(av+bu)^{2}+u^{2}+v^{2} &=& \lambda.
\end{aligned}
\end{equation}
This is (3.2) in \cite{CR84}. They further obtain the following rather explicit solutions to equations $(\ref{EinsteinCondition})$:
\begin{equation}\label{Solution}
\begin{split}
a &= c+\frac{1}{2}d+\frac{3}{2} \qquad (-1\leq c\leq 1),\cr
b &= c-\frac{1}{2}d+\frac{3}{2},\cr
u^{2} &= \frac{5}{2}-2(c+\frac{1}{2}d)^{2}, \qquad uv=-2+\frac{5}{2}c^{2},\cr
v^{2} &= \frac{5}{2}-2(c-\frac{1}{2}d)^{2},\cr
\lambda &= \frac{3}{2}(c+2)^{2},
\end{split}
\end{equation}
where $d=\pm\sqrt{1-c^{2}}$ and $c$ is related to $p$ and $q$ by
\begin{equation}\label{EquationForc}
\frac{3p}{q}=\frac{1-\frac{av}{bu}}{1+\frac{av}{bu}}.
\end{equation}
These are the equations  (3.4) and (3.5) in \cite{CR84}.

For each pair of co-prime non-negative integers $p, q$ with $3p\leq q$, a solution of $(\ref{EquationForc})$ satisfying $-1\leq c\leq -\frac{2}{\sqrt{5}}$ with $d=\sqrt{1-c^{2}}$ and the corresponding Einstein metric was obtained \cite{CR84}. Then in \cite{PP84}, Page and Pope showed that for each pair of such integers $p, q$, another solution of $(\ref{EquationForc})$ satisfying $\frac{2}{\sqrt{5}}\leq c\leq 1$ with $d=-\sqrt{1-c^{2}}$ actually gives a geometrically inequivalent Einstein metric, and furthermore there are exactly two geometrically inequivalent Einstein metrics on each $N^{pq0}$ among metrics of the form (\ref{InvariantMetric}). Their $\widetilde{\bf S}$-linear instability will be shown in \S2.2. However, in \cite{Nik04}, Nikonorov pointed out that the two Einstein metrics on $N^{130}$ obtained in \cite{CR84} and \cite{PP84} are isometric to each other. Moreover, he found a geometrically inequivalent invariant Einstein metric on $N^{130}$ and showed that there are exactly two geometrically inequivalent invariant Einstein metrics on $N^{130}$. The $\widetilde{\bf S}$-linear instability of the additional invariant Einstein metric will be shown in \S2.3.


\subsection{Instability of invariant Einstein metrics in \cite{CR84} and \cite{PP84}}   \label{AWI}

By using the Ricci curvature formulas (\ref{RicciCurvatureAloffWallach}), one easily obtains the scalar curvature of $g(\alpha, \beta, \gamma, \delta)$ as
\begin{equation*}
\begin{aligned}
s_{g(\alpha, \beta, \gamma, \delta)}=
&\frac{3}{2}\left(\frac{1}{\alpha}+\frac{1}{\gamma}+\frac{1}{\delta}\right)-\frac{1}{4}\left(\frac{\alpha}{\gamma\delta}
+\frac{\gamma}{\alpha\delta}+\frac{\delta}{\alpha\gamma}\right)\\
&-\frac{1}{4}q^{2}\frac{\beta}{\alpha^{2}}
-\frac{1}{16}(3p+q)^{2}\frac{\beta}{\gamma^{2}}-\frac{1}{16}(3p-q)^{2}\frac{\beta}{\delta^{2}}.
\end{aligned}
\end{equation*}
The volume of $(N^{pq0}, g(\alpha, \beta, \gamma, \delta))$, denoted by
$\Vol(g(\alpha, \beta, \gamma, \delta))$, is equal to $\Vol(Q)\alpha\beta^{\frac{1}{2}}\gamma\delta,$
where $\Vol(Q)$ denotes the volume of the space $N^{pq0}$ with the metric induced by $Q$. Then the normalized total scalar curvature of a metric $g(\alpha, \beta, \gamma, \delta)$ in $(\ref{InvariantMetric})$ is given by
\begin{equation*}
\begin{aligned}
\widetilde{\bf S}(g(\alpha, \beta, \gamma, \delta))
&=(\Vol(g(\alpha, \beta, \gamma, \delta)))^{\frac{2}{7}}s_{g(\alpha, \beta, \gamma, \delta)}\\
&=(\Vol(Q))^{\frac{2}{7}}(\alpha\beta^{\frac{1}{2}}\gamma\delta)^{\frac{2}{7}}\bigg[\frac{3}{2}\left(\frac{1}{\alpha}+\frac{1}{\gamma}+\frac{1}{\delta}\right)-\frac{1}{4}\left(\frac{\alpha}{\gamma\delta}
+\frac{\gamma}{\alpha\delta}+\frac{\delta}{\alpha\gamma}\right)\\
&\quad -\frac{1}{4}q^{2}\frac{\beta}{\alpha^{2}}
-\frac{1}{16}(3p+q)^{2}\frac{\beta}{\gamma^{2}}-\frac{1}{16}(3p-q)^{2}\frac{\beta}{\delta^{2}}\bigg]
\end{aligned}
\end{equation*}

By straightforward calculations, one has the following partial derivatives
\begin{equation*}
\begin{split}
\frac{\partial}{\partial\gamma}\widetilde{\bf S}(g(\alpha, \beta, \gamma, \delta)) &= \Vol(Q)^{\frac{2}{7}}\frac{1}{56}
(\alpha\beta^{\frac{1}{2}}\gamma\delta)^{\frac{2}{7}}\frac{1}{\alpha^{3}\gamma^{3}\delta^{3}}F_{3}(\alpha, \beta, \gamma, \delta),\\
\frac{\partial}{\partial\delta}\widetilde{\bf S}(g(\alpha, \beta, \gamma, \delta)) &= \Vol(Q)^{\frac{2}{7}}\frac{1}{56}
(\alpha\beta^{\frac{1}{2}}\gamma\delta)^{\frac{2}{7}}\frac{1}{\alpha^{3}\gamma^{3}\delta^{3}}F_{4}(\alpha, \beta, \gamma, \delta),\\
\end{split}
\end{equation*}
where


\begin{equation}\label{F3}
\begin{split}
F_{3}(\alpha, \beta, \gamma, \delta)=
&-60\alpha^{3}\gamma\delta^{3}+24(\alpha^{2}\delta^{3}+\alpha^{3}\delta^{2})\gamma^{2}+10\alpha^{4}\gamma\delta^{2}\\
&-18\alpha^{2}\gamma^{3}\delta^{2}+10\alpha^{2}\gamma\delta^{4}-4q^{2}\alpha\beta\gamma^{2}\delta^{3}\\
&+6(3p+q)^{2}\alpha^{3}\beta\delta^{3}-(3p-q)^{2}\alpha^{3}\beta\gamma^{2}\delta,
\end{split}
\end{equation}

\begin{equation}\label{F4}
\begin{split}
F_{4}(\alpha, \beta, \gamma, \delta)=
&-60\alpha^{3}\gamma^{3}\delta+24(\alpha^{2}\gamma^{3}+\alpha^{3}\gamma^{2})\delta^{2}+10\alpha^{4}\gamma^{2}\delta\\
&+10\alpha^{2}\gamma^{4}\delta-18\alpha^{2}\gamma^{2}\delta^{3}-4q^{2}\alpha\beta\gamma^{3}\delta^{2}\\
&-(3p+q)^{2}\alpha^{3}\beta\gamma\delta^{2}+6(3p-q)^{2}\alpha^{3}\beta\gamma^{3}.
\end{split}
\end{equation}

Let $g(\alpha_{0}, \beta_{0}, \gamma_{0}, \delta_{0})$ be a fixed but arbitrary invariant Einstein metric as in \cite{CR84} and \cite{PP84}. Then we investigate the stability of this invariant Einstein metric by varying the components of the  metric in $\mathfrak{m}_{3}\oplus\mathfrak{m}_{4}$. This keeps the variations  within the class of homogeneous metrics.
Accordingly consider the function
\begin{equation}
\widetilde{S}(t):= \widetilde{\bf S}(g(\alpha_{0}, \beta_{0}, \gamma_{0}+At, \delta_{0}+Bt)), \qquad t\geq0,
\end{equation}
where $A$ and $B$ are parameters.

\begin{prop}\label{CR-Instability}
There exist parameters $A$ and $B$ {\rm (}depending on $\alpha_{0}, \beta_{0}, \gamma_{0},$ and $\delta_{0}${\rm )} such that
\begin{equation}
\frac{d^{2}}{dt^{2}}\widetilde{S}(0)>0.
\end{equation}
\end{prop}

\begin{proof}
Since
\begin{equation*}
\begin{aligned}
\frac{d}{dt}\widetilde{S}(t)
&=A\, \frac{\partial\widetilde{\bf S}}{\partial\gamma}(\alpha_{0}, \beta_{0}, \gamma_{0}+At, \delta_{0}+Bt)+B\, \frac{\partial\widetilde{\bf S}}{\partial\delta}(\alpha_{0}, \beta_{0}, \gamma_{0}+At, \delta_{0}+Bt)\\
&=\frac{(\Vol(Q))^{\frac{2}{7}}(\alpha_{0}\beta_{0}^{\frac{1}{2}}(\gamma_{0}+At)(\delta_{0}+Bt))^{\frac{2}{7}}}{56\alpha_{0}^{3}(\gamma_{0}+At)^{3}(\delta_{0}+Bt)^{3}}
\,[A F_{3}(\alpha_{0}, \beta_{0}, \gamma_{0}+At, \delta_{0}+Bt)\\
&\quad +B F_{4}(\alpha_{0}, \beta_{0}, \gamma_{0}+At, \delta_{0}+Bt)],
\end{aligned}
\end{equation*}
and $F_{3}(\alpha_{0}, \beta_{0}, \gamma_{0}, \delta_{0})=F_{4}(\alpha_{0}, \beta_{0}, \gamma_{0}, \delta_{0})=0$, we have
\begin{equation*}
\begin{aligned}
\frac{d^{2}}{dt^{2}}\widetilde{S}(0)=
&\,\frac{(\Vol(Q))^{\frac{2}{7}}(\alpha_{0}\beta_{0}^{\frac{1}{2}}\gamma_{0}\delta_{0})^{\frac{2}{7}}}{56\alpha_{0}^{3}\gamma_{0}^{3}\delta_{0}^{3}}
\bigg[A^{2}\frac{\partial F_{3}}{\partial \gamma}(\alpha_{0}, \beta_{0}, \gamma_{0}, \delta_{0})\\
& +AB(\frac{\partial F_{3}}{\partial \delta}+\frac{\partial F_{4}}{\partial \gamma})(\alpha_{0}, \beta_{0}, \gamma_{0}, \delta_{0})
+B^{2}\frac{\partial F_{4}}{\partial \delta}(\alpha_{0}, \beta_{0}, \gamma_{0}, \delta_{0})\bigg].
\end{aligned}
\end{equation*}
Thus we only need to show that there exist $A$ and $B$ such that
$$A^{2}\frac{\partial F_{3}}{\partial \gamma}(\alpha_{0}, \beta_{0}, \gamma_{0}, \delta_{0})+AB(\frac{\partial F_{3}}{\partial \delta}+\frac{\partial F_{4}}{\partial \gamma})(\alpha_{0}, \beta_{0}, \gamma_{0}, \delta_{0})+B^{2}\frac{\partial F_{4}}{\partial \delta}(\alpha_{0}, \beta_{0}, \gamma_{0}, \delta_{0})>0.$$
For this it suffices to show that
\begin{equation}
\left[\left(\frac{\partial F_{3}}{\partial \delta}+\frac{\partial F_{4}}{\partial \gamma}\right)^{2}-4\left(\frac{\partial F_{3}}{\partial \gamma}\right)\left(\frac{\partial F_{4}}{\partial \delta}\right)\right](\alpha_{0}, \beta_{0}, \gamma_{0}, \delta_{0})>0.
\end{equation}

For the solution $(\alpha_{0}, \beta_{0}, \gamma_{0}, \delta_{0})$, from the equations in $(\ref{ChangeVariables})$, one can easily deduce that
\begin{equation}\label{ChangeVariables1}
q^{2}\beta_{0}=\frac{(av+bu)^{2}\alpha^{3}_{0}}{2\gamma_{0}\delta_{0}}, \quad (3p+q)^{2}\beta_{0}=\frac{2u^{2}\alpha_{0}\gamma_{0}}{\delta_{0}}, \quad (3p-q)^{2}\beta_{0}=\frac{2v^{2}\alpha_{0}\delta_{0}}{\gamma_{0}}.
\end{equation}
Then by substituting the first two equations in $(\ref{ChangeVariables})$, the last equation in $(\ref{EinsteinCondition})$, and equations in $(\ref{ChangeVariables1})$ into the partial derivatives of the functions $F_{3}$ and $F_{4}$ defined in $(\ref{F3})$ and $(\ref{F4})$, we obtain
\begin{align*}
&\quad\left[\left(\frac{\partial F_{3}}{\partial \delta}+\frac{\partial F_{4}}{\partial \gamma}\right)^{2}-4\left(\frac{\partial F_{3}}{\partial \gamma}\right)\left(\frac{\partial F_{4}}{\partial \delta}\right)\right](\alpha_{0}, \beta_{0}, \gamma_{0}, \delta_{0})\\
&=\alpha^{8}_{0}\gamma^{2}_{0}\delta^{2}_{0}[(-132b+144ba-132a+40+4b^{2}+4a^{2}-12\lambda+46u^{2}+46v^{2})^{2}\\
&\quad-4(-60a+48ab+48b+10-54b^{2}+10a^{2}-4\lambda+4u^{2})\cdot\\
&\quad(-60b+48ab+48a+10+10b^{2}-54a^{2}-4\lambda+4v^{2})]\\
&=32\alpha^{8}_{0}\gamma^{2}_{0}\delta^{2}_{0}(-392c^{4}-273c^{3}+812c^{2}+840c+168)\\
&=32\alpha^{8}_{0}\gamma^{2}_{0}\delta^{2}_{0}f(c)
\end{align*}
where
$$f(c):=-392c^{4}-273c^{3}+812c^{2}+840c+168.$$
In the second last step above, we have used the equations in $(\ref{Solution})$ and $d=\pm\sqrt{1-c^{2}}$. Since for all invariant Einstein metrics in \cite{CR84} and \cite{PP84} the parameter $c\in \big[-1, \frac{2}{\sqrt{5}}\big]\cup\big[\frac{2}{\sqrt{5}}, 1\big]$, in order to complete the proof, we only need to show that
$f(c)>0$ for such $c$.

By simple calculations, one can see that
$$f^{\prime\prime}(c)<0, \quad \text{for} -1\leq c\leq -0.85 \ \ \text{or} \ \ 0.85\leq c\leq 1.$$
It follows that
$$f^{\prime}(c)\leq f^{\prime}(-1)=-35<0 \quad \text{for} \ \ -1\leq c\leq -0.85,$$
and
$$f^{\prime}(c)\geq f^{\prime}(1)=77>0  \quad \text{for} \ \  0.85\leq c\leq 1.$$
Thus
$$f(c)\geq f(-0.85)>3>0  \quad \text{for} \ \ -1\leq c\leq -0.85,$$
and
$$f(c)\geq f(0.85)>1096>0  \quad \text{for} \ \ 0.85\leq c\leq 1.$$
In particular,
$$f(c)=-392c^{4}-273c^{3}+812c^{2}+840c+168>0, \quad \text{for} -1\leq c\leq -\frac{2}{\sqrt{5}} \ \ \text{or} \ \ \frac{2}{\sqrt{5}}\leq c\leq 1.$$
This completes the proof.
\end{proof}

Next we shall show that the invariant variations used in Proposition \ref{CR-Instability} above are actually divergence-free. Then
the $\widetilde{\bf S}$-linear instability of the invariant Einstein metrics follows immediately from Proposition $\ref{CR-Instability}$.

\begin{lem}  \label{divergencefree}
Let $(G/K, g)$ be a $G$-homogeneous Riemannian manifold of dimension $n$ with $G$ compact and $K \subset G$ closed. Let $Q$ be a fixed
bi-invariant metric on $G$ and use it to write $\g = \kf \perp \m$. Suppose that
\begin{equation}  \label{m-decomposition}
 \m = \m_1 \oplus \cdots \oplus \m_r
\end{equation}
is a $Q$-orthogonal decomposition of $\m$ into ${\rm Ad}(K)$-invariant summands.
Finally suppose that $g$ and $G$-invariant symmetric $2$-tensor $h$ are given by
\begin{eqnarray*}
        g  & = &  a_1 Q| \m_1 \oplus \cdots \oplus a_r Q | \m_r, \,\,\,  a_i > 0  \\
        h  & = &  c_1 Q| \m_1 \oplus \cdots \oplus c_r Q | \m_r, \,\,\,  c_i \in \rr.
\end{eqnarray*}
Then $\delta_g h = 0$.
\end{lem}

\begin{proof} We identify $\mathfrak{m}$ with the tangent space of $G/H$ at $[H]$ as usual.
 Let $\{X_{1}, \cdots, X_{n}\}\subset\mathfrak{m}$ be a orthonormal basis with respect to $g$, and extend them to Killing vector fields
 in a neighborhood of the base point $[H]$. Then we have at $[H]$
\begin{align*}
(\delta_{g}h)(X_{j}) & =-\sum^{n}_{i=1}(\nabla_{X_{i}}h)(X_{i}, X_{j})\\
                     & =-\sum^{n}_{i=1}\big(X_{i}(h(X_{i}, X_{j}))-h(\nabla_{X_{i}}X_{i}, X_{j})-h(X_{i}, \nabla_{X_{i}}X_{j})\big)\\
                     & =-\sum^{n}_{i=1}\big(h(X_{i}, [X_{i}, X_{j}])-h(X_{i}, \nabla_{X_{i}}X_{j})-h(\nabla_{X_{i}}X_{i}, X_{j})\big)\\
                     & =\sum^{n}_{i=1}\big(h(X_{i}, \nabla_{X_{j}}X_{i})+h(\nabla_{X_{i}}X_{i}, X_{j})\big)\\
                     & =\sum^{n}_{i, k=1}g(\nabla_{X_{j}}X_{i}, X_{k})h(X_{i}, X_{k})+
                         \sum^{n}_{i, k=1}g(\nabla_{X_{i}}X_{i}, X_{k}) h(X_k, X_j)
\end{align*}
where in the third equality above we used the $G$-invariance of $h$.

We next use the fact that covariant derivatives involving Killing vector fields on a homogeneous
 Riemannian manifold can be expressed entirely in terms of Lie brackets (see e.g. Lemma 7.27 in \cite{Bes87}).
After some simplification and replacing brackets for vector fields with the negative  of the corresponding Lie brackets in
$\g$,  we obtain
\begin{equation}  \label{div-formula}
 (\delta_g h)(X_j) =  \sum_{i, k}  h(X_j, X_k) g([X_k, X_i], X_i) -  \sum_{i}  h([X_j, X_i]_{\m}, X_i)
\end{equation}
where $[\cdot, \cdot]_{\m}$ denotes the $Q$-orthogonal projection of the bracket onto $\m$.

Let $\{ e_q^{(\ell)}, 1 \leq \ell \leq r, \, 1 \leq q \leq d_{\ell} := \dim \m_{\ell} \}$ be a $Q$-orthonormal
basis of $\m$ adapted to the decomposition (\ref{m-decomposition}). The corresponding adapted $g$-orthonormal
basis is then given by $ X_q^{(\ell)} := \frac{1}{\sqrt{a_{\ell}}} e_q^{(\ell)}$. We examine separately the two
sums in (\ref{div-formula}). Let $X_j = X_q^{(\ell)} \in \m_{\ell}$.

The first sum is then equal to
$$ \frac{c_{\ell}}{a_{\ell}} \sum_{i=1}^r \sum_{\alpha = 1}^{d_i} \, a_i \,
       Q\left(\left[\frac{e_q^{(\ell)}}{\sqrt{a_{\ell}}},  \frac{e_{\alpha}^{(i)}}{\sqrt{a_i}}\right], \frac{e_{\alpha}^{(i)}}{\sqrt{a_i}} \right)
  =  \frac{c_{\ell}}{a_{\ell}^{\frac{3}{2}}} \sum_{i=1}^r \sum_{\alpha = 1}^{d_i} \,
       Q\left(\left[e_q^{(\ell)},  e_{\alpha}^{(i)} \right], e_{\alpha}^{(i)} \right) = 0  $$
since $Q$ is bi-invariant.

Similarly, the second sum is equal to
$$ \frac{1}{\sqrt{a_{\ell}}} \sum_{i=1}^r \sum_{\alpha = 1}^{d_i} \, c_i \,
       Q\left(\left[e_q^{(\ell)},  \frac{e_{\alpha}^{(i)}}{\sqrt{a_i}}\right], \frac{e_{\alpha}^{(i)}}{\sqrt{a_i}} \right)
  =  \frac{1}{\sqrt{a_{\ell}}}  \,\sum_{i=1}^r \frac{c_i}{a_i} \,\sum_{\alpha = 1}^{d_i} \,
       Q\left(\left[e_q^{(\ell)},  e_{\alpha}^{(i)} \right], e_{\alpha}^{(i)} \right) = 0  $$
again by the bi-invariance of $Q$.
\end{proof}

\begin{rmk}
Note that in the above Lemma, the ${\rm Ad}(K)$-invariant summands $\m_i$ are not assumed to be
irreducible or pairwise inequivalent. This will be important in the next subsection.
\end{rmk}

\begin{prop}\label{CR-linear-instability}
The invariant Einstein metrics on $N^{pq0}$ where $(p, q) \neq (1, 3)$ are $\widetilde{\bf S}$-linearly unstable.
\end{prop}
\begin{proof}
Let $g(\alpha_{0}, \beta_{0}, \gamma_{0}, \delta_{0})$ be a fixed but arbitrary invariant Einstein metric on $N^{pq0}\neq N^{130}$.
Up to isometry any such metric is diagonal with respect to the decomposition (\ref{m-decomposition}).
Moreover, with the parameters $A$ and $B$ obtained in Proposition $\ref{CR-Instability}$,  $A\cdot(Q\vert_{\mathfrak{m}_{3}})\oplus B\cdot(Q\vert_{\mathfrak{m}_{4}})$ is an Ad$_H$-invariant symmetric bilinear form on $\m$, and so it induces an $\SU(3)$-invariant
symmetric 2-tensor $h$ on $N^{pq0}$.

By Proposition $\ref{CR-Instability}$, the second variation of $\widetilde{\bf S}$ at $g(\alpha_{0}, \beta_{0}, \gamma_{0}, \delta_{0})$
is strictly positive along $h$, i.e., $\widetilde{\bf S}^{\prime\prime}_{g(\alpha_{0}, \beta_{0}, \gamma_{0}, \delta_{0})}(h, h)>0$. If we
replace $h$ by its trace-free part given by $h_{0}=h-\frac{2}{7}\big(\frac{A}{\gamma_{0}}+\frac{B}{\delta_{0}}\big)g(\alpha_{0}, \beta_{0}, \gamma_{0}, \delta_{0})$, since the normalized total scalar curvature functional is homothety invariant, we have
$$\widetilde{\bf S}^{\prime\prime}_{g(\alpha_{0}, \beta_{0}, \gamma_{0}, \delta_{0})}(h_{0}, h_{0})=\widetilde{\bf S}^{\prime\prime}_{g(\alpha_{0}, \beta_{0}, \gamma_{0}, \delta_{0})}(h, h)>0.$$
But Lemma $\ref{divergencefree}$ implies that $h$ is divergence-free, and so $h_{0}$ is a TT-tensor. Thus $h_{0}$ is a $\widetilde{\bf S}$-linearly unstable direction.
\end{proof}


\subsection{Instability of the invariant Einstein metric on $N^{130}$ in \cite{Nik04}}  \label{AWII}
In order to complete the proof of Theorem $\ref{Aloff-Wallach}$, we only need to check the $\widetilde{\bf S}$-linear instability of
the invariant Einstein metric on $N^{130}$ found by Nikonorov in \cite{Nik04}.

For $N^{130}$ the irreducible sub-representations $\mathfrak{m}_{1}$ and $\mathfrak{m}_{4}$ in $(\ref{LieAlgDec})$ are isomorphic to each other. Thus there are $\SU(3)$-invariant metrics on $N^{130}$ that are not block diagonal with respect to the decomposition of the isotropy representation in $(\ref{LieAlgDec})$. In \cite{Nik04}, Nikonorov showed that the two block diagonal (with respect to the decomposition in $(\ref{LieAlgDec})$) $\SU(3)$-invariant Einstein metrics on $N^{130}$ obtained in \cite{CR84} and \cite{PP84} are isometric to each other. He also found a geometrically distinct $\SU(3)$-invariant Einstein metric on $N^{130}$ that is not block diagonal with respect to the decomposition in $(\ref{LieAlgDec})$.

Since $\m_1$ and $\m_4$ are equivalent as $H_{1,3}$-representations, there is a whole circle's worth of different ways of decomposing
$\m_1 \perp \m_4$ as $ W_1 \perp W_2$, where $W_i$ are subspaces of $\m_1 \perp \m_4$ which are isomorphic as $H_{1,3}$-representations
to $\m_1 \approx \m_4$.  Thus  Nikoronov considered the $1$-parameter family of ${\rm Ad}(H_{1,3})$-invariant
irreducible decompositions of the isotropy representation on $N^{130}$ given by
\begin{equation}\label{LieAlgDecInNik}
\mathfrak{m}=\mathfrak{p}_{1}\oplus\mathfrak{p}_{2}\oplus\mathfrak{p}_{3}\oplus\mathfrak{p}_{4}
\end{equation}
where $\mathfrak{p}_{1}={\rm span}(Y_{1}, Y_{2}), \mathfrak{p}_{2}={\rm span}(Y_{3}, Y_{4}), \mathfrak{p}_{3}=\mathfrak{m}_{3}={\rm span}(X_{4}, X_{5}), \mathfrak{p}_{4}=\mathfrak{m}_{2}={\rm span}(Z)$,
$$Y_{1}=-\cos(\alpha)(2X_{2})-\sin(\alpha)(2X_{7}), \qquad Y_{2}=-\cos(\alpha)(2X_{1})-\sin(\alpha)(2X_{6}),$$
$$Y_{3}=\sin(\alpha)(2X_{2})-\cos(\alpha)(2X_{7}), \qquad Y_{4}=\sin(\alpha)(2X_{1})-\cos(\alpha)(2X_{6}),$$
$\alpha\in\mathbb{R}$, and $X_{1}, X_{2}, X_{4}, X_{5}, X_{6}, X_{7}$, and $Z$ are as given in $(\ref{LieAlgebraBasis})$.
He showed that for a suitably chosen value $\alpha\in\mathbb{R}$, this additional invariant Einstein metric is diagonal with respect to the corresponding decomposition in $(\ref{LieAlgDecInNik})$. Therefore, we will consider in the following the invariant metrics on $N^{130}$ of the form
\begin{equation}
g(x_{1}, x_{2}, x_{3}, x_{4})=x_{1}Q^{\prime}\vert_{\mathfrak{p}_{1}}\oplus x_{2}Q^{\prime}\vert_{\mathfrak{p}_{2}}\oplus x_{3}Q^{\prime}\vert_{\mathfrak{p}_{3}}\oplus
x_{4}Q^{\prime}\vert_{\mathfrak{p}_{4}},
\end{equation}
where $Q^{\prime}$ is the multiple of the Killing form of $\mathfrak{su}(3)$ given by $Q^{\prime}(X, Y)=-\frac{1}{2}\tr(XY)$ for $X, Y\in \mathfrak{su}(3)$.
For this family of metrics the scalar curvature formula is given in \cite{Nik04} as
\begin{equation}
\begin{aligned}
s_{g(x_{1}, x_{2}, x_{3}, x_{4})}=
&\frac{12}{x_{1}}+\frac{12}{x_{2}}+\frac{12}{x_{3}}+\frac{6a}{x_{4}}-\frac{3-3a}{2}\bigg(\frac{x_{4}}{x^{2}_{1}}+\frac{x_{4}}{x^{2}_{2}}\bigg)\\
&-2\bigg(\frac{x_{1}}{x_{2}x_{3}}+\frac{x_{2}}{x_{1}x_{3}}+\frac{x_{3}}{x_{1}x_{2}}\bigg)
-3a\bigg(\frac{x_{1}}{x_{2}x_{4}}+\frac{x_{1}}{x_{2}x_{4}}+\frac{x_{4}}{x_{1}x_{2}}\bigg),
\end{aligned}
\end{equation}
where $a=\sin^{2}(2\alpha)$.

The Einstein equations were then considered in three different cases: $a=0, a=1,$ and $0<a<1$. When $a=0$, the  Einstein metrics obtained in \cite{CR84} and \cite{PP84} were recovered. When $a=1$, two solutions of the Einstein  equations were found. Approximate values of these solutions are
\begin{equation}\label{Nik-solution1}
(x_{1}, x_{2}, x_{3}, x_{4})\approx(5.67352, 1.09220, 5.50695, 5.72906),
\end{equation}
and
\begin{equation}\label{Nik-solution2}
(x_{1}, x_{2}, x_{3}, x_{4})\approx(1.09220, 5.67352, 5.50695, 5.72906).
\end{equation}
However, these two solutions give rise to isometric invariant Einstein metrics. When $0<a<1$, no new Einstein metrics were obtained.

We will now show that the new Einstein metric obtained in the $a=1$ case is unstable.
The normalized total scalar curvature of $g(x_{1}, x_{2}, x_{3}, x_{4})$ is
\begin{align*}
\widetilde{\bf S}(g(x_{1}, x_{2}, x_{3}, x_{4}))=
& \,\Vol(Q^{\prime})^{\frac{2}{7}}(x^{2}_{1}x^{2}_{2}x^{2}_{3}x_{4})^{\frac{1}{7}}
\bigg[\frac{12}{x_{1}}+\frac{12}{x_{2}}+\frac{12}{x_{3}}+\frac{6}{x_{4}}\\
&-2\bigg(\frac{x_{1}}{x_{2}x_{3}}+\frac{x_{2}}{x_{1}x_{3}}+\frac{x_{3}}{x_{1}x_{2}}\bigg)
-3\bigg(\frac{x_{1}}{x_{2}x_{4}}+\frac{x_{1}}{x_{2}x_{4}}+\frac{x_{4}}{x_{1}x_{2}}\bigg)\bigg],
\end{align*}
where $\Vol(Q^{\prime})$ is the volume of $N^{130}$ with respect to the metric induced by $Q^{\prime}\vert_{\mathfrak{m}}$.
\begin{prop}
At the solution given by $(\ref{Nik-solution1})$, we have
\begin{equation}
\frac{\partial^{2}}{\partial x^{2}_{2}}\,\widetilde{\bf S}(g(x_{1}, x_{2}, x_{3}, x_{4}))>0.
\end{equation}
\end{prop}

\begin{proof}
One easily computes that
\begin{equation}\label{FristDerivativeNik}
\begin{aligned}
\frac{\partial}{\partial x_{2}}\widetilde{\bf S}(g(x_{1}, x_{2}, x_{3}, x_{4}))=
&\Vol(Q^{\prime})^{\frac{2}{7}}(x^{2}_{1}x^{2}_{2}x^{2}_{3}x_{4})^{\frac{1}{7}}\frac{2}{7x_{2}}\bigg[\frac{12}{x_{1}}-\frac{30}{x_{2}}+\frac{12}{x_{3}}
+\frac{6}{x_{4}}+5\frac{x_{1}}{x_{2}x_{3}}\\
&-9\frac{x_{2}}{x_{1}x_{3}}+5\frac{x_{3}}{x_{1}x_{2}}+\frac{15}{2}\frac{x_{1}}{x_{2}x_{4}}
-\frac{27}{2}\frac{x_{2}}{x_{1}x_{4}}+\frac{15}{2}\frac{x_{4}}{x_{1}x_{2}}\bigg].
\end{aligned}
\end{equation}

At the solution given by $(\ref{Nik-solution1})$, the second order partial derivative with respect to $x_{2}$ is
\begin{equation}\label{SecondDerivativenNik}
\begin{aligned}
\frac{\partial^{2}}{\partial x^{2}_{2}}\widetilde{\bf S}(g(x_{1}, x_{2}, x_{3}, x_{4}))=
& \Vol(Q^{\prime})^{\frac{2}{7}}(x^{2}_{1}x^{2}_{2}x^{2}_{3}x_{4})^{\frac{1}{7}}\frac{2}{7x_{2}}\bigg[\frac{30}{x^{2}_{2}}-5\frac{x_{1}}{x^{2}_{2}x_{3}}
-9\frac{1}{x_{1}x_{3}}\\
& -5\frac{x_{3}}{x_{1}x^{2}_{2}}-\frac{15}{2}\frac{x_{1}}{x^{2}_{2}x_{4}}
-\frac{27}{2}\frac{1}{x_{1}x_{4}}-\frac{15}{2}\frac{x_{4}}{x_{1}x^{2}_{2}}\bigg]
\end{aligned}
\end{equation}

Because the first order partial derivative in $(\ref{FristDerivativeNik})$ vanishes at this solution, we have
\begin{equation}
-\bigg(5\frac{x_{1}}{x_{3}}+5\frac{x_{3}}{x_{1}}+\frac{15}{2}\frac{x_{1}}{x_{4}}+\frac{15}{2}\frac{x_{4}}{x_{1}}\bigg)
=12\frac{x_{2}}{x_{1}}-30+12\frac{x_{2}}{x_{3}}+6\frac{x_{2}}{x_{4}}-9\frac{x^{2}_{2}}{x_{1}x_{3}}-\frac{27}{2}\frac{x^{2}_{2}}{x_{1}x_{4}}.
\end{equation}
Substituting this into $(\ref{SecondDerivativenNik})$ and factoring $\frac{1}{x^{2}_{2}}$ out, we obtain
\begin{eqnarray*}
\frac{\partial^{2}}{\partial x^{2}_{2}}\widetilde{\bf S}(g(x_{1}, x_{2}, x_{3}, x_{4}))
&=&\Vol(Q^{\prime})^{\frac{2}{7}}(x^{2}_{1}x^{2}_{2}x^{2}_{3}x_{4})^{\frac{1}{7}}\frac{2}{7x^{3}_{2}}
\bigg[12\frac{x_{2}}{x_{1}}\\
& &+\bigg(12-18\frac{x_{2}}{x_{1}}\bigg)\frac{x_{2}}{x_{3}}+\bigg(6-27\frac{x_{2}}{x_{1}}\bigg)\frac{x_{2}}{x_{4}}\bigg].
\end{eqnarray*}
This is strictly positive for $(x_{1}, x_{2}, x_{3}, x_{4})\approx(5.67352, 1.09220, 5.50695, 5.72906)$, since $x_{1} > 5.5, x_{2}<1.1,$ and therefore $\frac{x_{2}}{x_{1}}<\frac{1}{5}$.
\end{proof}

For the solution $(x_{1}, x_{2}, x_{3}, x_{4})\approx(1.09220, 5.67352, 5.50695, 5.72906)$, similarly one can show that $\frac{\partial^{2}}{\partial x^{2}_{1}}\widetilde{\bf S}(g(x_{1}, x_{2}, x_{3}, x_{4}))$ is strictly positive.

As in Proposition $\ref{CR-linear-instability}$, this implies that the invariant Einstein metric obtained in \cite{Nik04} is $\widetilde{\bf S}$-linearly unstable. Together with Proposition $\ref{CR-linear-instability}$ this completes the proof of Theorem $\ref{Aloff-Wallach}$.


\section{\bf Instability of Einstein metrics on the Stiefel manifolds}   \label{Stiefel}

The Stiefel manifold $V_{2}(\mathbb{R}^{n+1})= \frac{\SO(n+1)}{\SO(n-1)}$ $(n\geq 3)$ may be viewed as a
principal circle bundle over the real oriented Grassmannian $\frac{\SO(n+1)}{\SO(n-1)\SO(2)}$, which is an
irreducible Hermitian symmetric space (with second Betti number $b_2 = 1$). Sagle \cite{Sa70} constructed an
invariant Einstein metric on $V_{2}(\mathbb{R}^{n+1})$ which is now known to be unique (up to isometry and homothety) among all
$\SO(n+1)$-invariant metrics \cite{Ker98}, except when $n=3$. Its relevance for us is that it can be viewed as the regular
Sasaki Einstein metric determined by the base considered as a Fano Einstein manifold with the symmetric
metric scaled so that its scalar curvature equals $2n+2$. In this section, we will show that this Einstein metric
is $\widetilde{\bf S}$-linearly unstable, and therefore $\nu$-linearly unstable. Additionally, in Example
$\ref{hyperquadric}$ in \S $\ref{instab-conf}$, we will show that this Einstein metric is also $\nu$-linearly
unstable along conformal variation directions. When $n=3$, $\SO(4)/\SO(2)$ is diffeomorphic to $S^2 \times S^3$,
so the product metric is a second Einstein metric which is not isometric to the Sasaki Einstein metric. The product
metric is of course also $\widetilde{\bf S}$-linearly unstable.

We will follow the notation in \cite{Ker98}. Embedding $\SO(n-1)$ into $\SO(n+1)$ as
$\SO(n-1)\cong\begin{bmatrix} Id_{2} & 0 \\ 0 & \SO(n-1)\end{bmatrix}\subset \SO(n+1)$ gives rise to the Stiefel manifold
$V_{2}(\mathbb{R}^{2})$ as a quotient space $\frac{SO(n+1)}{SO(n-1)}$. On the Lie algebra level, the embedding is $\mathfrak{so}(n-1)\cong\begin{bmatrix} 0 & 0 \\ 0 & \mathfrak{so}(n-1)\end{bmatrix}\subset \mathfrak{so}(n+1)$.
We then choose the ${\rm Ad}_{\SO(n-1)}$-invariant complement $\mathfrak{p}=\mathfrak{so}(n-1)^{\perp}$
(with respect to the Killing form). The isotropy representation of $\SO(n-1)$ on $\mathfrak{p}$ can be decomposed into irreducible sub-representations as $\mathfrak{p}=\mathfrak{p}_{0}\oplus\mathfrak{p}_{1}\oplus\mathfrak{p}_{2}$, where $\mathfrak{p}_{0}=$span$\{E_{12}\}$, $\mathfrak{p}_{i}=$span$\{E_{j, 2+i} \vert 1\leq i\leq n-1\}$ for $j=1, 2$,
and $E_{ij}$ denotes the matrix with 1 in the $(i,j)$-entry,  $-1$ in the $(j,i)$-entry, and zeros everywhere else.

Let $Q^{\prime}$ be the multiple of the  Killing form of $\mathfrak{so}(n+1)$ given by
$Q^{\prime}(X, Y)=-\frac{1}{2}\tr(XY)$ for $X, Y\in \mathfrak{so}(n+1)$, and choose
$Q^{\prime}|_{\mathfrak{p}}$ as the background metric. Then we consider $\SO(n+1)$-invariant metrics on
$\frac{\SO(n+1)}{\SO(n-1)}$ induced by
\begin{equation*}
g(x_{0}, x_{1}, x_{2})=x_{0}Q^{\prime}|_{\mathfrak{p}_{0}}\oplus x_{1}Q^{\prime}|_{\mathfrak{p}_{1}}\oplus x_{2}Q^{\prime}|_{\mathfrak{p}_{2}}
\end{equation*}
for $x_{0}, x_{1}, x_{2}>0$.
Recall the scalar curvature formula in section 4 of \cite{Ker98} as
\begin{equation*}
s_{g(x_{0}, x_{1}, x_{2})}=(n-1)\left(\frac{n-1}{x_{1}}+\frac{n-1}{x_{2}}+\frac{1}{x_{0}}\right)
-\frac{n-1}{2}\left(\frac{x_{1}}{x_{2}x_{0}}+\frac{x_{2}}{x_{1}x_{0}}+\frac{x_{0}}{x_{1}x_{2}}\right).
\end{equation*}
By considering variations of this scalar curvature function, Kerr showed that $x_{1}=x_{2}$ and $x_{0}=\frac{2(n-1)}{n}x_{1}$ give the unique SO$(n+1)$-invariant Einstein metric up to diffeomorphisms and homotheties. In particular, we will consider the Einstein metric with $x_{0}=2(n-1), x_{1}=x_{2}=n$.

Now the normalized total scalar curvature of $g(x_{0}, x_{1}, x_{2})$ is
\begin{eqnarray*}
\widetilde{\bf S}(g(x_{0}, x_{1}, x_{2}))
                      &=& \Vol(Q^{\prime})^{\frac{2}{2n-1}}(x_{0}x^{n-1}_{1}x^{n-1}_{2})^{\frac{1}{2n-1}}
                      \bigg[(n-1)\left(\frac{n-1}{x_{1}}+\frac{n-1}{x_{2}}+\frac{1}{x_{0}}\right)\\
                      & & -\frac{n-1}{2}\left(\frac{x_{1}}{x_{2}x_{0}}+\frac{x_{2}}{x_{1}x_{0}}+\frac{x_{0}}{x_{1}x_{2}}\right)\bigg],
\end{eqnarray*}
where $\Vol(Q^{\prime})$ is the volume of the Stiefel manifold with the metric induced by $Q^{\prime}$. Its first partial derivative with respect to $x_{1}$ is
\begin{eqnarray*}
\frac{\partial}{\partial x_{1}}\widetilde{\bf S}(g(x_{0}, x_{1}, x_{2}))
&=&\Vol(Q^{\prime})^{\frac{2}{2n-1}}(x_{0}x^{n-1}_{1}x^{n-1}_{2})^{\frac{1}{2n-1}}\frac{(n-1)}{(2n-1)x_{1}}
\bigg[-\frac{n(n-1)}{x_{1}}\\
& &+(n-1)\left(\frac{n-1}{x_{2}}+\frac{1}{x_{0}}\right)-\frac{(3n-2)x_{1}}{2x_{2}x_{0}}
+\frac{n}{2}\left(\frac{x_{2}}{x_{1}x_{0}}+\frac{x_{0}}{x_{1}x_{2}}\right)\bigg].
\end{eqnarray*}
Then at the Einstein metric with $x_{0}=2(n-1), x_{1}=x_{2}=n$, we have the second derivative with respect to $x_{1}$ as
\begin{eqnarray*}
& &\frac{\partial^{2}}{\partial x^{2}_{1}}\widetilde{\bf S}(g(x_{0}, x_{1}, x_{2}))|_{(2(n-1), n, n)}\\
&=& \Vol(Q^{\prime})^{\frac{2}{2n-1}}(x_{0}x^{n-1}_{1}x^{n-1}_{2})^{\frac{1}{2n-1}}\frac{(n-1)}{(2n-1)x_{1}}
\bigg[\frac{n(n-1)}{x^{2}_{1}}-\frac{3n-2}{2x_{2}x_{0}}\\
& &+\frac{n}{2}\left(-\frac{x_{2}}{x^{2}_{1}x_{0}}-\frac{x_{0}}{x^{2}_{1}x_{2}}\right)\bigg]\bigg|_{(x_{0}, x_{1}, x_{2})=(2(n-1), n, n)}\\
&=& \Vol(Q^{\prime})^{\frac{2}{2n-1}}(2(n-1)n^{2n-2}_{1})^{\frac{1}{2n-1}}\frac{(n-1)[(n-3)(2n^2-2n+1)+1]}{2(2n-1)(n-1)n^3}\\
&\geq& \Vol(Q^{\prime})^{\frac{2}{2n-1}}(2(n-1)n^{2n-2}_{1})^{\frac{1}{2n-1}}\frac{(n-1)}{2(2n-1)(n-1)n^3}>0,
\end{eqnarray*}
for $n\geq3$.

As in Proposition $\ref{CR-linear-instability}$, together with Lemma $\ref{divergencefree}$,
this  implies the $\widetilde{\bf S}$-linear instability of the invariant Sasaki Einstein metric on
$V_{2}(\mathbb{R}^{n+1})$ with $n\geq3$.


\section{\bf Instability from conformal deformations} \label{instab-conf}

A second source of instability for the $\nu$-functional comes from conformal deformations of the
Einstein metric in question. A sufficient condition for instability is that the smallest nonzero eigenvalue of
the Laplace-Beltrami operator is less than $2 \Lambda$, where $\Lambda$ denotes the Einstein constant
\cite{CH15}. In fact, by \cite{CHI04}, provided that the Einstein manifold $(M, g)$ is not the constant
curvature sphere, the operator $\mathscr{S}$ given by
\begin{equation}
{\mathscr S} u := - {\rm Hess}_g u + (\Delta u) g + \Lambda u g
\end{equation}
is injective and maps eigenfunctions of the Laplace-Beltrami operator with eigenvalue $\lambda$ to divergence-free
symmetric $2$-tensors which are eigentensors of the Lichnerowicz Laplacian with the same eigenvalue.

When the Einstein manifold is a homogeneous space $(G/K, g)$ where $G$ is a semisimple compact Lie
group, $K$ is a closed subgroup, and $g$ is induced by the negative of the Killing form
$Q_G$ of $G$, then $L^2(G/K, g)$ is a Hilbert space direct sum of the irreducible finite-dimensional
unitary representations of $G$ which are of class $1$ with respect to $K$, with multiplicity equal to the
dimension of the subspace of $K$-fixed vectors \cite{MU80}. Furthermore, the eigenvalues are given by the
Casimir constants $Q_G( \lambda, \lambda + 2 \delta )$ of the irreducible class $1$ representations,
where $\lambda$ is the dominant weight of the representation, and $2\delta$ is the sum of the positive roots of $G$.
By abuse of notation we have used $Q_G$ to denote also the inner product induced by the Killing form on the dual of the
chosen real Cartan subalgebra in $\g$. Obviously, in the above we can replace $g$ by any negative multiple of
the Killing form. (Note that $Q_G$ is negative definite on $\g$ but positive definite on
the Cartan subalgebra.)

If, on the other hand, the Einstein manifold lies in the canonical variation
(see \cite{Bes87} pp. 252-255) of the Killing form metric along a closed intermediate subgroup
$K \subset H \subset G$, then the spectrum of the Laplacian can be determined using the results
in \cite{BB90}, which improves upon the work in \cite{BeBo82}.

\begin{example} \label{NK-S3S3}
Let $G = S^3 \times S^3 \times S^3$ and $K$ be the image of the diagonally embedded  $S^3$  in $G$.
The Killing form metric is well-known to be nearly K\"ahler, and the dimension of the associated space of
real Killing spinors is $1$. We will show that this Einstein metric is $\nu$-unstable.

For convenience we will take the metric $g$ on $G$ to be the product of the normalized Killing form $Q^{\prime}$
of $\SU(2)$, which is that multiple of the Killing form such that the maximal root of $\SU(2)$ has length $-\sqrt{2}$.
Since $Q = - 4 Q^{\prime}$, using Corollary 1.7  and Table IV in \cite{WZ85}, one deduces that the Einstein constant of
$g$ is $\Lambda = \frac{5}{3}$. (Note that $g$ induces three times $Q^{\prime}$ on the diagonal subalgebra and
the isotropy representation of $G/K$ consists of two copies of the adjoint representation of $\SU(2)$.)

Next we determine the irreducible class $1$ unitary representations of $S^3 \times S^3 \times S^3$ relative to
the diagonal subgroup. The irreducible unitary representations of $S^3 \times S^3 \times S^3$
consist of external tensor products $\rho_1 \hat{\otimes} \rho_2 \hat{\otimes} \rho_3$ of irreducible
unitary representations of the individual factors. If only one $\rho_i$ is non-trivial, then the representation
remains irreducible upon restriction to the diagonal subgroup and so cannot be of class $1$.
If exactly two of the $\rho_i$ are non-trivial and equal to the $2$-dimensional vector representation of $S^3$, then
the Clebsch-Gordon formula shows that a $1$-dimensional trivial summand appears upon restriction to the diagonal
subgroup. Hence by permuting the $S^3$ factors we obtain three inequivalent class $1$ irreducible representations
of $S^3 \times S^3 \times S^3$ all having the same Casimir constant of $2 \cdot \frac{3}{2} = 3$. Since this is
less than $2 \Lambda = \frac{10}{3}$, $\nu$-instability has been established.
\end{example}

It seems appropriate to recall here the following lemma which we will use repeatedly later in this section.

\begin{lem} \label{compareCasimir}
Let $\g$ be a complex semisimple Lie algebra with Killing form $Q$. Let $\lambda_1, \lambda_2$ denote the dominant weights of two
irreducible $($finite-dimensional$)$ complex representations of $\g$. Assume that for each simple root $\alpha$
of $\g$ we have
$$  \frac{2 Q(\lambda_1, \alpha)}{Q(\alpha, \alpha)} \geq \frac{2 Q(\lambda_2, \alpha)}{Q(\alpha, \alpha)}. $$
Then the corresponding Casimir constants satisfy
$$ Q(\lambda_1, \lambda_1 + 2\delta) \geq Q(\lambda_2, \lambda_2 + 2\delta)$$
with equality iff $\lambda_1 = \lambda_2$.
\end{lem}

\begin{rmk}
Applying the above lemma to Example \ref{NK-S3S3} we see that the first eigenspace of the Laplacian
has dimension $12$.
\end{rmk}

Using the classification theorem of Butruille \cite{Bu05}  for strict nearly K\"ahler simply connected homogeneous
$6$-manifolds, one obtains

\begin{prop}  \label{NK-appl}
The only $\nu$-stable strict nearly K\"ahler simply connected homogeneous $6$-manifold is $S^6 = {\rm G}_2/\SU(3)$
with the round metric.
\end{prop}

\begin{proof}
Recall that a strict nearly K\"ahler structure is one that is not K\"ahler. The classification theorem of Butruille
states that, up to homothety, the only simply connected homogeneous strict nearly K\"ahler $6$-manifolds are
$S^6 = {\rm G}_2/\SU(3), (\SU(2)\times \SU(2) \times \SU(2))/ \Delta \SU(2), \cc{\rm P}^3 = \Sp(2)/(\Sp(1) \times \U(1)), $
and $\SU(3)/T^2$, each equipped with a unique invariant nearly K\"ahler structure.

In the first case, the nearly K\"ahler metric is the constant curvature metric, which is $\widetilde{\bf S}$-stable.
The second case is treated in Example \ref{NK-S3S3} above. For the last two cases, the nearly K\"ahler metric
lies in the canonical variation of the Riemannian submersions given by the twistor fibrations
$$ \Sp(2)/(\Sp(1) \U(1))  \longrightarrow \Sp(2)/(\Sp(1)\times \Sp(1)) = S^4 $$
$$   \SU(3)/T^2 \longrightarrow \SU(3)/{\rm S}(\SU(2)\U(1)) = \cc{\rm P}^2$$
equipped with the metrics induced by the negative of the Killing form of $G$. The Einstein metrics are given
by 9.72 of \cite{Bes87} and the first graph of Fig. 9.72 there. The Fubini-Study metric on $\cc{\rm P}^3$
is $\widetilde{\bf S}$-stable \cite{Koi80} so the strict nearly K\"ahler one must be given by the local minimum in the canonical variation.

For the last case, the fiber and base metrics are Einstein. The Einstein constant $\Lambda_B$ of the base is $\frac{1}{2}$
since we are using the Killing form metric on a symmetric space (see Corollary 1.6 in \cite{WZ85}). The fibers
are $\SU(2)/\U(1)$ and hence are symmetric as well. But the Killing form of $\SU(3)$ restricts to $\frac{3}{2}$
times the Killing form of $\SU(2)$ by page 583 of \cite{WZ85}. So the Einstein constant $\Lambda_F$ of the fibers is
$\frac{2}{3} \cdot \frac{1}{2} = \frac{1}{3}$.  It follows that $\Lambda_B - 2 \Lambda_F = -\frac{1}{6} < 0$
and a destablizing TT-tensor is given, for example, by Theorem 1.1 in \cite{WW18}.
\end{proof}

\begin{rmk} It is actually known that the Killing form metric on $\SU(3)/T^2$ is a local minimum
for the normalized scalar curvature functional on the space of $\SU(3)$-invariant metrics. This can be
checked by directly computing the Hessian of the normalized scalar curvature function at the Killing form
metric.  Since $b_2(\SU(3)/T^2) = 2$, it follows from \cite{CHI04} that the three invariant K\"ahler Einstein
metrics on it are also $\nu$-unstable.

Analogous computations for $\SU(n+1)/T^n$ show that the Killing form metric is also $\nu$-unstable.
\end{rmk}

Consider next the situation in which we have a circle bundle
\begin{equation}  \label{bundle-over-symmetric}
  F=\U(1) = (H \cdot \U(1))/H   \longrightarrow M = G/H  \stackrel{\pi}{\longrightarrow} B= G/(H \cdot \U(1))
\end{equation}
where the base is an irreducible compact Hermitian symmetric space of dimension $2m$.
$B$ is simply connected and has second Betti number equal to $1$.  The above fibration becomes a Riemannian
submersion with totally geodesic fibers if we give $G/H$ and $G/(H \cdot \U(1))$ the normal metrics induced by $Q_G$.
We shall denote these respectively by $g$ and $\check{g}$, and denote the induced metrics on the fibers by $\hat{g}$.
In particular the base metric is K\"ahler-Einstein and has Ricci curvature $\frac{1}{2}$.

The canonical variation of $g$ introduced in Chapter 9.G of \cite{Bes87} is the $1$-parameter family of metrics
$$ g_t =  t^2 \hat{g} + \check{g} $$
on the total space $M$, where in the above definition the horizontal and vertical distributions of the Riemannian submersion are
used implicitly. There is a unique choice of $t^2$ (indeed $t^2 = \frac{2m}{m+1}$) which makes $g_t$ into an Einstein metric
with Einstein constant $\Lambda=\frac{m}{2m+2}$. If we multiply $g_t$ by $\frac{1}{4m+4}$, then the resulting metric
would have Einstein constant $2m$ and the submersed metric on the base would have Einstein constant $2m+2$.
In other words, the rescaled metric would be Sasaki Einstein. Since stability properties are independent of homothety, those
of $g_t$ are the same as those of the corresponding Sasaki Einstein metric.

Simply connected Sasaki Einstein manifolds are spin and admit non-trivial real Killing spinors. So we shall
take care in the following to ensure that $G/H$ is also simply connected. Let $t_*^2$ denote the special value
$\frac{2m}{m+1}$. It follows that the Einstein metric $g_{t_*}$ admits at least two linearly independent Killing spinors.
We shall show below that in some cases the Einstein metric $g_{t_*}$ is $\nu$-unstable by exhibiting an eigenvalue of the
Laplacian which is less than $2\Lambda_{g_{t_*}} = \frac{m}{m+1}$.

To do this we will use the results in \cite{BeBo82} and \cite{BB90}, which we recall briefly below. Let $\Delta_t$ and $\Delta_v$
denote respectively the Laplacian of $g_t$ and the vertical Laplacian of the Riemannian submersion (\ref{bundle-over-symmetric})
(with metric $g_1$). Because of the totally geodesic property, all fibers of our fibration are isometric, and the vertical Laplacian
is just the collection of Laplacians of the fibers (for the metrics $\hat{g}$). Then $\Delta_1$ is the Laplacian of
the Killing form metric, whose eigenvalues can be found using representation theory. We have the relation
$$ \Delta_t = \Delta_1 + \left(\frac{1}{t^2} -1 \right) \Delta_v.$$
Note that in the totally geodesic situation, $\Delta_1$ and $\Delta_v$ commute. The operator
$\Delta_v$ is not elliptic, however, but has discrete spectrum, and its eigenvalues can have infinite multiplicities.
The crucial fact for us is the
\begin{thm} $($\cite{BeBo82} Theorem 3.6, \cite{BetPi13} Remark 3.3$)$  \label{simult-eigen}
$L^2(M, g_1)$ has a Hilbert space basis consisting of simultaneous eigenfunctions of $\Delta_1$ and $\Delta_v$.
\end{thm}

 It follows from this that every eigenvalue of $\Delta_t$ is the sum of an eigenvalue of $\Delta_1$ and
$\frac{1}{t^2} -1$ times an eigenvalue of $\Delta_v$. It is not completely straight-forward to decide which
combinations of eigenvalues occur in general, but this has been worked out in \cite{BB90}.

In our situation, we will assume $m>1$; otherwise $B= S^2$ and $M$, being of dimension $3$, must have
constant curvature and so is $\widetilde{\bf S}$-stable. Then
$t_*^2 > 1$ and $\frac{1}{t_*^2} - 1 = -\left(\frac{m-1}{2m}  \right) < 0.$
Let $h_0$ denote the usual metric on $\U(1) = S^1$ so that it has circumference $2\pi$. Suppose the Killing form metric
induces on the $\U(1)$ fibers in  (\ref{bundle-over-symmetric}) the metric $a h_0$ where $a > 0$. Then since the eigenvalues
of $\Delta_v$ are of the form $\frac{1}{a} \ell^2$ where $\ell \in \zz$, it follows that the eigenvalues of $\Delta_{t_*}$
are of the form
\begin{equation} \label{eigenvalueform}
\lambda +  \left(\frac{1}{t_*^2} - 1\right) \frac{\ell^2}{a} =   \lambda - \left(\frac{m-1}{2m}\right) \frac{ \ell^2}{a} \leq \lambda,
\end{equation}
where $\lambda$ is an eigenvalue of $\Delta_1$.

In order to apply the results in \cite{BB90}, we need to write $M = G/H$ as  $G/H \times_L \U(1)$
where $L:=\U(1)$ acts freely on the right of $P=G/H=M$ and isometrically on the left of $F = \U(1)$.
Note that the metric on $M$ is $g_{t_*}$ and the metric on $F$ is $t_*^2 a h_0$. On the other hand,
the principal bundle $p: P \longrightarrow B$ is the projection of $G/H$ onto $B=G/K$ where both spaces
are equipped with the normal metric induced by $Q_G$, and so $L$ indeed acts via isometries
of this metric.  Now in the proof of the results in \cite{BB90}, the authors employ a separate canonical
variation along the fibers of this Riemannian submersion, which, when combined with Cheeger's trick, kills off
the metric along $L$ as the variation parameter tends to infinity. In the limit we then get the eigenvalues of
$\Delta_{t_*}$ expressed as the sum of eigenvalues of the horizontal Laplacian of the fibration $p$ and
``corresponding" eigenvalues of $(F, t_*^2 a h_0)$.  Here ``corresponding"  means that the action of the group
$L$ on the irreducible summands of the eigenspaces in $L^2(P, Q_G)$ and $L^2(F, t_*^2 a h_0)$ must be the same.
Finally, note that the eigenvalues of the horizontal Laplacian of the fibration $p$ can be written as a
difference of an eigenvalue of the Laplacian of $Q_G$ and an eigenvalue of the vertical Laplacian corresponding
to the metric $a h_0$. This gives back the form  (\ref{eigenvalueform}) of the eigenvalue together with
the additional information as to which $\ell$ can occur for a given $\lambda$.

\medskip

\noindent{\bf Observation}: The above discussion implies that if we can find an irreducible unitary representation
of $G$ that is of class $1$ relative to $H$ on which $L$ acts non-trivially and if this representation has
a Casimir constant $\leq \frac{m}{m+1}$ then the Einstein metric $g_{t*}$ is $\nu$-unstable. Furthermore, if inequality
holds for the Casimir constant, then the same conclusion holds without having to check whether the action of $L$ is trivial or not.

\begin{example}  \label{hyperquadric}
Let $G=\SO(m+2), K= \SO(m) \times \SO(2)$, and $H= \SO(m)$ with $m\geq 3$. Then $L = \SO(2)$ and $B= G/K$
is the hyperquadric of complex dimension $m$. Note that $G/H$ is simply connected. The vector representation
$\rho_{m+2}$ of $\SO(m+2)$ on $\cc^{m+2}$ has a fixed point set of complex dimension $2$ when restricted to
$\SO(m)$. $L$ acts on the right of $\cc^{m+2}$ via the usual
representation of $\U(1)$ by rotations ($\ell = 1$). The element $i$ of the Lie algebra of $\U(1)$ corresponds to the
matrix in $\so(m+2)$ consisting of zeros everywhere except for a single $2 \times 2$ block in the lower right hand
corner given by
$$ \left( \begin{array}{rr}
             0   &  -1 \\
             1  &  0
          \end{array} \right).
$$
This matrix has length $2m$ with respect to $Q_G$, and so the constant $a = 2m$.
The Casimir constant of $\rho_{m+2}$ is $\frac{m+1}{2m} < \frac{m}{m+1}$
since $Q_G = 2(m+2 -2) Q_G^{\prime}$ where $Q_G^{\prime}$ is that multiple of the negative of the Killing form
so that the maximal root has length $-\sqrt{2}$. (See pp. 583-586 of \cite{WZ85} for more details.)
The corresponding eigenvalue of $\Delta_{t_*}$ is
$$\frac{m+1}{2m} - \frac{m-1}{2m} \frac{1}{2m}.$$
Since the multiplicity of the eigenvalue $\frac{m+1}{2m}$ in $L^2(\SO(m+2)/\SO(m))$ is $2(m+2)$ we obtain a
$2(m+2)$-dimensional positive definite subspace for the second variation of the $\nu$-functional that is
orthogonal to the unstable direction we found in \S \ref{Stiefel}.
\end{example}

\begin{example}  \label{E6}
Let $G = {\rm E}_6$,  $H = \Spin(10)$,  and $K = (\Spin(10) \times \U(1))/\Delta(\zz/4)$. Then
$B=G/K$ is a Hermitian symmetric space of dimension $2m = 32$, and $G/H$ is simply connected.
(But $G/K$ is not effective since the center $\zz/3$ of ${\rm E}_6$ (lying in the $\U(1)$ factor in $K$)
is the ineffective kernel.) The Einstein constant $\Lambda_{g_{t_*}}$ is equal to $\frac{16}{34}$.

Let $\pi_{\lambda}$ be one of the two lowest dimensional irreducible unitary representations  of ${\rm E}_6$,
and $\lambda$ be its dominant weight. The complex dimension of $\pi_{\lambda}$ is $27$, and by Table 25, p. 203, of
\cite{Dyn52}, upon restriction to $\Spin(10)$ it decomposes as $\rho_{10} \oplus \Delta^{+}_{10} \oplus \I$
where $\rho_{10}$ is the vector representation of $\Spin(10)$, $\Delta^{+}_{10}$ is the positive spin representation,
and $\I$ denotes a trivial one-dimensional representation. Hence $\pi_{\lambda}$ is of class $1$ with respect to
$\Spin(10)$. If we picked the other lowest dimensional representation, which is contragedient to $\pi_{\lambda}$,
then in the decomposition the $+$ spin representation would be replaced by the $-$ spin representation.
Now $Q_{{\rm E}_6} = 24 Q^{\prime}_{{\rm E}_6}$ so from Table III, p. 586 of \cite{WZ85},
$Q_G(\lambda, \lambda + 2\delta) = \frac{1}{24} \cdot \frac{52}{3} = \frac{13}{18} < \frac{16}{17} = 2 \Lambda_{g_{t_*}}.$
So there is no need to determine the action of $L = \U(1)$ on the right of $\pi_{\lambda}$. We obtain
a $2 \cdot 27 = 54$-dimensional subspace of divergence-free symmetric $2$-tensors on which the second
variation of the $\nu$-functional is positive definite.
\end{example}

\begin{example}  \label{E7}
Let $G= {\rm E}_7$, $H= {\rm E}_6$ and $K = {\rm E}_6 \cdot \U(1)$ where $K$ is the quotient of
${\rm E}_6 \times \U(1)$ by the diagonally embedded $\zz/3$. (The center of ${\rm E}_6$ is
$\zz/3$.) $G/H$ is simply connected.  The dimension of $G/K = B$ is $2m= 54$ and so the Einstein
constant $\Lambda$ for $g_{t_*}$ is $\frac{27}{56}$. $G/K$ again is not effective, but can be made so
by dividing by the center of ${\rm E}_7$, which is $\zz/2$.

We consider the lowest dimensional non-trivial irreducible representation $\pi_{\lambda}$ of ${\rm E}_7$,
which is  of dimension $56$.
By Table 25, p. 204 of \cite{Dyn52}, upon restriction to ${\rm E}_6$, $\pi_{\lambda}$ decomposes as
$2 \I \oplus \rho$, where $\rho$ is the real irreducible representation of ${\rm E}_6$ corresponding to one of the
$27$-dimensional irreducible complex representations of ${\rm E}_6$. So $\pi_{\lambda}$ is of class $1$ with respect to
${\rm E}_6$ with fixed point set of dimension $2$. Because $Q_G = 36 Q_G^{\prime}$, the Casimir constant
$Q_G(\lambda, \lambda + 2 \delta) = \frac{1}{36} \cdot \frac{57}{2} = \frac{57}{72} < 2\Lambda = \frac{27}{28}$
(see Table III, p. 586 of \cite{WZ85}). So again we do not need to determine the action of $\U(1)$ on this
irreducible summand in $L^2(G/H)$ and we obtain a $2 \cdot 56$-dimensional subspace on which the second variation of the
$\nu$-functional is positive definite.
\end{example}

\begin{example} \label{CGr-2planes}
Let $G = \SU(p+2)$, $H=\SU(p) \times \SU(2)$, and $K = {\rm S}(\U(p)\times \U(2))$ with $p \geq 2$. Then $G/H$ is simply
connected. Note that $G/K$ has the distinction of being the only Hermitian symmetric space that is also quaternionic
symmetric. Its dimension is $2m = 4p$, so the Einstein constant $\Lambda_{t_*} = \frac{p}{2p+1}$.

Let $\mu_k$ denote the vector representation of $\SU(k)$ on $\cc^k$.  We claim that $\Lambda^2 \mu_{p+2}$,
which is irreducible, is of class $1$ relative to $H$. This follows from the calculation
\begin{equation}  \label{rep-decompose}
\Lambda^2 \mu_{p+2}|\, \SU(p)\times \SU(2) = \Lambda^2 \mu_p \hat{\otimes} \I \oplus \I \hat{\otimes} \Lambda^2 \mu_2
      \oplus \mu_p \hat{\otimes} \mu_2
\end{equation}
where $\I$ denotes the $1$-dimensional trivial representation and $\hat{\otimes}$ denotes the external
tensor product. Since $\mu_2$ has dimension $2$ and the determinants in $\SU(2)$ equal to $1$, we get a
single trivial summand upon restriction, provided $p > 2$.

Let $\lambda$ denote the dominant weight of $\mu_{p+2}$. Since $Q_{\SU(p+2)} = 2(p+2) Q_{\SU(p+2)}^\prime$,
using Table III of \cite{WZ85}, we have
$$ Q_G(\lambda, \lambda + 2\delta) = \frac{1}{2(p+2)} \cdot 2p \cdot \frac{p+3}{p+2} = \frac{p(p+3)}{ (p+2)^2} <
       2 \Lambda_{g_{t_*}} = \frac{2p}{2p+1}.$$
So again it is unnecessary to determine the action of $\U(1) = L$ on $\Lambda^2 \mu_{p+2}$. We obtain a
$(p+2)(p+1)/2$-dimensional subspace on which the second variation of the $\nu$-functional is positive definite.

Notice that when $p=2$, $G/H = \SU(4)/(\SU(2)\times \SU(2)) = \SO(6)/\SO(4)$, which we analysed in Example (\ref{hyperquadric}).
In this special case, we have an additional trivial summand coming from the $\Lambda^2  \mu_p$ in  (\ref{rep-decompose}),
and so the multiplicity of $\Lambda^2 \mu_{p+2}$ in $L^2(G/H)$ is doubled, which is consistent with the analysis
in Example (\ref{hyperquadric}).
\end{example}


\section{\bf Low-dimensional homogeneous Einstein spaces and Sasaki Einstein spaces} \label{homog}

In this section we will first apply the results in the earlier sections and in \cite{WW18} to determine the stability
of low-dimensional simply connected compact homogeneous Einstein manifolds. Given such a manifold we will write it
in the form $G/K$ where $G$ is compact, connected, semisimple, and $K$ is a closed subgroup of $G$. We also assume that $G$
acts almost effectively on $G/K$. These assumptions are not too restrictive, since the isometry group of a compact
Riemannian manifold is compact, and the semisimple part of the identity component of a transitive Lie group acting
on a simply connected closed manifold also acts transitively on it.

The $\widetilde{\bf S}$ and $\nu$-linear stability of the symmetric metric on compact symmetric spaces
was analysed by Koiso \cite{Koi80} and Cao-He \cite{CH15}. We shall therefore assume that $(G, K)$ is not a symmetric
pair unless otherwise stated. Also recall that any product Einstein metric with positive scalar curvature is $\widetilde{\bf S}$-linearly unstable.

We shall begin with dimension five, since Jensen \cite{J69} proved that all simply connected homogeneous $4$-manifolds
are symmetric.

\smallskip
\noindent{\bf I. Dimension five}
\smallskip

The classification of simply connected compact homogeneous Einstein $5$-manifolds was given in
\cite{ADF96}. There are only two non-symmetric cases: the Stiefel manifold $\SO(4)/\SO(2)$, and the family
$(\SU(2) \times \SU(2))/U_{k,l}$ where $k, l$ are relatively prime integers not both equal to $1$, and
$U_{k, l}$ is the circle embedded by $e^{i \theta} \mapsto (e^{i k \theta}, e^{i l \theta})$. The unique
$\SU(2) \times \SU(2)$-invariant Einstein metric is in fact of Riemannian submersion type over
$S^2 \times S^2$. It is not Sasaki Einstein, and by Corollary 1.3 in \cite{WW18}, it is $\widetilde{\bf S}$-linearly
unstable. The case of the Stiefel manifold $\SO(4)/\SO(2)$ actually corresponds to the $(k, l) = (1, 1)$ case of the
above infinite family. There are two invariant Einstein metrics on this space. One is the Sasaki Einstein metric, which is
$\widetilde{\bf S}$-linearly unstable, and the other is the product metric, which is also $\widetilde{\bf S}$-linearly
unstable.

The symmetric $5$-manifolds are $S^5$ (stable), $S^3 \times S^2$ ($\widetilde{\bf S}$-linearly unstable),
and $\SU(3)/\SO(3)$, which is neutrally linearly stable, i.e., $\nu$-linearly stable and the kernel of
the second variation operator contains a symmetric $2$-tensor orthogonal to the orbit of the diffeomorphism
group.

\medskip
\noindent{\bf II. Dimension six}
\medskip

The classification of simply connected compact homogeneous Einstein metrics in dimension $6$ is as yet
incomplete. The only open case is that of $S^3 \times S^3$ with a left-invariant metric. The remaining possibilities
are classified in \cite{NR03}. In the same paper, the authors showed that if the left-invariant Einstein metric
has an additional circle of isometries acting by right translations, then up to isometries and homotheties it
must be the product metric or the strict nearly K\"ahler metric induced by the Killing form on
$(\SU(2) \times \SU(2) \times \SU(2))/ \Delta \SU(2)$. Quite recently, this result has been improved
in \cite{BCHL18} to allow the same conclusion as long as $S^3\times S^3 = G/K$ with $K \neq \zz/2$.

The stability of the strict nearly K\"ahler metrics was dealt with in Proposition \ref{NK-appl} in section
\ref{instab-conf}. The only non-symmetric case is that of $\cc\pp^3 = \Sp(2)/(\Sp(1) \times \U(1))$
with the Ziller metric. This metric is $\widetilde{\bf S}$-linearly unstable as it lies in the canonical variation
of the Fubini-Study metric on $\cc\pp^3$, viewed as a Riemannian submersion with totally
geodesic fibers over the self-dual Einstein space $\hh\pp^1= S^4$.

The symmetric cases are all $\widetilde{\bf S}$-linearly unstable except for $S^6 $ and $\cc\pp^3$, which are
both stable.

\medskip
\noindent{\bf III. Dimension seven}
\medskip

The seven-dimensional simply connected compact homogeneous Einstein manifolds were classified in
\cite{Nik04}. Except for $S^7$ the symmetric cases are all product manifolds, and hence are $\widetilde{\bf S}$-linearly
unstable. As for the non-symmetric cases, those which are not product manifolds consist of
\begin{enumerate}
\item[$($1$)$] the Aloff-Wallach spaces $N_{k, l}$, with $k, l$ relatively prime integers; \\
\item[$($2$)$]  the circle bundles over $S^2 \times S^2 \times S^2$;  \\
\item[$($3$)$]  the circle bundles over $\cc\pp^2 \times S^2$;  \\
\item[$($4$)$]  the Jensen squashed $7$-sphere; \\
\item[$($5$)$]  the Stiefel manifold $SO(5)/\SO(3)$; \\
\item[$($6$)$]  the isotropy irreducible space $\Sp(2)/\SU(2)$ where the embedding of $\SU(2)$ is via the
        irreducible $4$-dimensional symplectic representation.
\end{enumerate}

The first case is covered by Theorem \ref{Aloff-Wallach}, proved in \S \ref{AWI} and \ref{AWII}. The second and
third cases have $\widetilde{\bf S}$-coindex of at least $2$ and $1$ respectively by results
in \cite{WW18}. The metric in case $4$ lies in the canonical variation of the Riemannian submersion given by
the Hopf fibration. It is clearly $\widetilde{\bf S}$-linearly unstable since the round metric on $S^7$ is stable.
The Euclidean cone of the Jensen sphere has $\Spin(7)$ holonomy, i.e., the Jensen metric is nearly ${\rm G}_2$ with only
a $1$-dimensional space of real Killing spinors. The fifth case is discussed in \S \ref{Stiefel} and in Example
$\ref{hyperquadric}$ in \S $\ref{instab-conf}$. The nature of the
last case remains open.

The discussions in I - III above completes the proof of Theorem \ref{lowdim}.

\medskip

 One of the two isometry classes of $\SU(3)$-invariant Einstein metrics on $N_{1,1}$ is $3$-Sasakian and fits
into the more general context of regular $3$-Sasakian manifolds. We refer the reader to Chapter 13 of \cite{BG08}
for background about this family of spaces. It turns out that regular $3$-Sasakian manifolds are given by certain principal $\SO(3)$ or $\Sp(1)$
bundles over a quaternionic K\"ahler manifold with positive scalar curvature. (One gets an $\Sp(1)$ bundle only
when the quaternionic K\"ahler manifold is quaternionic projective space.) The prevailing conjecture is that
the only quaternionic K\"ahler manifolds with positive scalar curvature are the quaternionic symmetric spaces.
This conjecture has been proved in dimensions $4$ \cite{Hit81} and $8$  \cite{PS91}.

The $3$-Sasakian metric makes the bundle projection into a Riemannian submersion with totally geodesic fibres.
By looking at the canonical variation of this Riemannian submersion, it follows immediately that there is a
second Einstein metric on the principal bundle that is not isometric to the $3$-Sasakian metric.
(This fact was independently observed by B\'erard Bergery and S. Salamon.) In fact the second Einstein metric is
always a local minimum in the canonical variation, see e.g., Theorem 3.4.1 in \cite{BG99}.  So this more general
viewpoint explains the existence of the second Einstein metric and its $\widetilde{\bf S}$-linear instability, and applies
in particular to $N_{1, 1}$. The $\widetilde{\bf S}$-linear instability of the $3$-Sasakian metric, which cannot
be detected from the canonical variation, can be explained by Corollary 1.7 in \cite{WW18}, since it is also the
Sasaki Einstein metric on the circle bundle over one of the homogeneous K\"ahler Einstein metrics on $\SU(3)/T$,
which has $b_2 > 1$.

The same argument works for the regular $3$-Sasakian manifold over the complex Grassmannian $\SU(p+2)/({\rm S}(\U(p) \times \U(2))$,
which are the only Hermitian symmetric spaces with a quaternionic K\"ahler structure. For the twistor spaces of the other compact
quaternionic symmetric spaces, the second Betti number is $1$, so their instability is at present unclear. We have therefore deduced

\begin{cor}  \label{3-Sasakian} The two Einstein metrics lying in the canonical variation of the
regular $3$-Sasakian fibration $\SO(3) \longrightarrow \SU(p+2)/({\rm S}( \U(p) \times \Delta {\rm S}^1) )
\longrightarrow \SU(p+2)/({\rm S}(\U(p) \times \U(2))$, $p \geq 1$, are both $\widetilde{\bf S}$-linearly unstable.
\end{cor}



\end{document}